\documentclass[a4paper,11pt]{article}
\usepackage{latexsym}
\usepackage{amssymb}
\usepackage{enumerate}
\usepackage{amsfonts}
\usepackage{amsmath}
\usepackage[dvips]{graphicx}
\usepackage{epsf}

\newenvironment{@abssec}[1]{%
    \if@twocolumn

      \section*{#1}%
    \else

      \vspace{.05in}\footnotesize
      \parindent .2in
 {\upshape\bfseries #1. }\ignorespaces
    \fi}

    {\if@twocolumn\else\par\vspace{.1in}\fi}

\newenvironment{keywords}{\begin{@abssec}{\keywordsname}}{\end{@abssec}}

\newenvironment{AMS}{\begin{@abssec}{\AMSname}}{\end{@abssec}}

\newcommand\keywordsname{Key words}
\newcommand\AMSname{AMS subject classifications}
\newcommand\AMname{AMS subject classification}
\newtheorem{theorem}{Theorem}
 \newtheorem{lemma}[theorem]{Lemma}
 \newtheorem{corollary}[theorem]{Corollary}
 \newtheorem{proposition}[theorem]{Proposition}
\newtheorem{remark}[theorem]{Remark}

\def\theequation{\arabic{section}.\arabic{equation}}
 \def\thetheorem{\arabic{section}.\arabic{theorem}}

\def\theequation{\arabic{section}.\arabic{equation}}
 \def\thetheorem{\arabic{section}.\arabic{theorem}}

\title{Stationary isothermic surfaces in Euclidean 3-space
\thanks{This research was partially supported by the ERC grant 335079 and the Spanish MINECO grant SEV-2011-0087, and by Grants-in-Aid
for Scientific Research (B) ($\sharp$ 20340031 and $\sharp$ 26287020) and for Challenging Exploratory Research ($\sharp$ 25610024) of Japan Society for the Promotion of Science. The first and third author have also
been supported by the Gruppo Nazionale per l'Analisi Matematica, la Probabilit\`a e le loro Applicazioni (GNAMPA) of the italian Istituto Nazionale di Alta Matematica (INdAM).} }

\author{Rolando Magnanini\thanks{Dipartimento di Matematica U.~Dini,
Universit\` a di Firenze, viale Morgagni 67/A, 50134 Firenze, Italy.
({\tt magnanin@math.unifi.it}).} 
, Daniel Peralta-Salas\thanks{Instituto de Ciencias Matem\'aticas, Consejo Superior de Investigaciones Cient\'ificas, 28049 Madrid, Spain.
({\tt dperalta@icmat.es}).}
  \ and Shigeru Sakaguchi\thanks{Research Center for Pure and Applied Mathematics,
Graduate School of  Information Sciences, Tohoku
University, Sendai, 980-8579,  Japan.
({\tt sigersak@m.tohoku.ac.jp}).}}

\begin{document}

\maketitle

\begin{abstract}
  Let  $\Omega$ be a domain in $\mathbb R^3$ with  $\partial\Omega = \partial\left(\mathbb R^3\setminus \overline{\Omega}\right)$, where $\partial\Omega$ is unbounded and connected, and let $u$ be the solution of  the Cauchy problem for the heat equation  $\partial_t u= \Delta u$ over $\mathbb R^3,$ where the initial data is  
the characteristic function of the set $\Omega^c = \mathbb R^3\setminus \Omega$. We show that, if there exists a stationary isothermic surface $\Gamma$ of $u$ with $\Gamma \cap \partial\Omega = \varnothing$, then both
$\partial\Omega$ and $\Gamma$ must be either parallel planes or  co-axial circular cylinders . This theorem completes the classification of stationary isothermic surfaces in the case that $\Gamma\cap\partial\Omega=\varnothing$ and $\partial\Omega$ is unbounded. To prove this result, we establish a similar theorem for {\it uniformly dense domains } in $\mathbb R^3$, a notion that was  introduced by Magnanini, Prajapat \& Sakaguchi in~\cite{MPS2006tams}. In the proof, we use methods from the theory of surfaces with constant mean curvature, combined with a careful analysis of certain asymptotic expansions and a surprising connection with the theory of transnormal functions. 
\end{abstract}



\pagestyle{plain}
\thispagestyle{plain}
\markboth{R. MAGNANINI, D. PERALTA-SALAS, and S. SAKAGUCHI}{Stationary isothermic surfaces in Euclidean 3-space}

\pagestyle{plain}
\thispagestyle{plain}

\section{Introduction}
\label{introduction}
Let $\Omega$ be a domain in $\mathbb R^N$ with  $N \ge 2.$  Consider the unique bounded solution $u = u(x,t)$ of the Cauchy  problem:
\begin{equation}
\label{cauchy}
\partial_t u=\Delta u\ \mbox{ in }\ \mathbb R^N \times (0, +\infty)\quad\mbox{ and }\ u = {\mathcal X}_{\Omega^c}\ \mbox{ on }\ \mathbb R^N \times \{0\},
\end{equation}
where ${\mathcal X}_{\Omega^c}$ denotes the characteristic function of the set $\Omega^c = \mathbb R^N \setminus \Omega$. 
A hypersurface $\Gamma$ in $\mathbb R^N$ is called a {\it stationary isothermic surface} of $u$ 
if at each time $t$ the solution $u$ remains constant on $\Gamma$ (a constant depending on $t$).  The following problem was raised 
in \cite{MPS2006tams}:
$$
\mbox{Classify all the domains $\Omega$ having a stationary isothermic surface.}
$$
\par
For $N=2$ the answer is easy: $\partial\Omega$ is either a circle, a straight line or a couple
of parallel straight lines (see \cite{MPS2006tams}).
One can also easily show that, if $\partial\Omega$ is a sphere, a hyperplane or, up to rescalings, any spherical cylinder
$\mathbb{S}^{k-1}\times\mathbb{R}^{N-k}$, $2\le k\le N-1$, then every level surface of $u$ is a stationary isothermic surface. 
Another interesting example is a helicoid $\mathcal H$ in $\mathbb R^3$. If $\Omega$ is a domain in $\mathbb R^3$ with $\partial\Omega=\mathcal H$, then $\mathcal H$ is a stationary isothermic surface of $u$ (see \cite[p. 4824]{MPS2006tams}).

In order to study this problem, the notion of {\it uniformly dense domains} was introduced in~\cite{MPS2006tams}. Let $B(x,r)$ be the open ball of positive radius $r$ and center $x \in \mathbb R^N$ and define the density
\begin{equation}
\label{average}
\rho(x,r) = \frac {|\Omega\cap B(x,r)|}{|B(x,r)|},
\end{equation}
where $|\Omega\cap B(x,r)|$ and $|B(x,r)|$ denote the $N$-dimensional Lebesgue measure of the sets $\Omega\cap B(x,r)$ and $B(x,r)$, respectively.
As defined in \cite{MPS2006tams}, $\Omega$ is {\it uniformly dense} in the hypersurface $\Gamma$ if and only if there exists $r_0 \in (\mbox{dist}(\Gamma, \partial\Omega), +\infty]$ such that,
for  every fixed $r \in (0, r_0)$,  the function
$x \mapsto \rho(x,r)$ is constant on $\Gamma$.\footnote{This assumption 
can be relaxed: it would be enough to assume that there exist two functions
 $r(x)\in (\mbox{dist}(\Gamma, \partial\Omega), +\infty)$ for $x \in \Gamma$ and $\rho_0(r) \ge 0$ for $r > 0$ such that $\rho(x,r) = \rho_0(r)$ provided that $x \in \Gamma$ and  $0 < r < r(x)$.}  
Thus, if $\partial\Omega = \partial\left(\mathbb R^N \setminus \overline{\Omega}\right)$ and $\Omega$ is uniformly dense in $\Gamma$, it is clear that
any point $x \in \Gamma$ must have the same distance, say $R$, from $\partial\Omega$, i.e. $\Gamma$ and $\partial\Omega$ are parallel hypersurfaces. 

In fact, stationary isothermic surfaces and uniformly dense domains are connected by the formula
\begin{equation}
\label{representation of the solution of the Cauchy problem}
1-u(x,t)=2\,(4\pi t)^{-N/2}\int_0^\infty |\Omega\cap B(x,2\sqrt{t} s)|\, se^{-s^2} ds,
\end{equation}
that can be easily derived by applying the integration by parts formula and a change of variables to the representation formula (2.4) in \cite[page 4825]{MPS2006tams}.  Hence by the arguments used in \cite[Theorem 1.1]{MPS2006tams}, we see that
$\Gamma$ is a stationary isothermic surface of $u$ if and only if $\Omega$ is uniformly dense in $\Gamma$ with $r_0 =\infty$. Heuristically, this means that a stationary isothermic surface collects local and global information about the set $\Omega$, since that formula also informs us that the short and large time behavior of $u$ are respectively linked to the behavior of $\rho(x,r)$ for small and large values of $r$. The results presented in this paper only use the local information about $\Omega$, since they rely on the behavior of $\rho(x,r)$
for small values of $r$, that is when $r_0<\infty$. (Nevertheless, we believe that 
the assumption $r_0=\infty$, besides simplifying some arguments as in \cite{MPS2006tams}, 
is necessary to attempt a classification in general dimension of uniformly dense domains and, of course, stationary isothermic surfaces.)

The case, where $\Gamma \subset \partial\Omega$ and $\Omega$ is uniformly dense in $\Gamma$, is considered in \cite[Theorems 1.2,  1.3, and 1.4]{MPS2006tams}. In particular it is shown that, if $N=3$ and $\partial\Omega$ is connected, then $\partial\Omega$ must be a sphere, a circular cylinder, or a minimal surface. Also, if $\partial\Omega$ is a complete embedded minimal surface of finite total curvature in $\mathbb R^3$, then it must be a plane. 

The case, where $\Gamma\cap\partial\Omega=\varnothing,\ \partial\Omega$ is bounded, and $\Omega$ is uniformly dense in $\Gamma$,  is studied in \cite[Theorem 3.6]{MPS2006tams} and \cite[Theorem 1.2]{MS2013mmas}, where in particular it is shown that $\partial\Omega$ must be a sphere, if it is connected.  Here the boundedness of $\partial\Omega$ enables us to use Alexandrov's sphere theorem to reach the conclusion.
The case, where $\Gamma\cap\partial\Omega=\varnothing,\ \partial\Omega$ is an entire graph over $\mathbb R^{N-1}$, and $\Gamma$ is a stationary isothermic surface of $u$, is analyzed in \cite[Theorem 2.3]{MS2012jde} and in \cite[Theorem 2]{Sa2013springer}, and it is shown that $\partial\Omega$ must be a hyperplane under some additional conditions on $\partial\Omega$. 
In both of these cases the global conditions on $\partial\Omega$ play a key role to reach the conclusions. Therefore, in order  to classify all the cases in which $\Gamma\cap\partial\Omega=\varnothing,\ \partial\Omega$ is unbounded, and $\Omega$ is uniformly dense in $\Gamma$ {\it with no global assumptions on $\partial\Omega$},  a new approach must be developed. In the present paper we show that such a classification is actually possible in $\mathbb R^3$. 

To complete the picture, we mention that an interesting generalization of uniformly dense domains  --- the so called {\it $K$-dense sets} ---
is considered in \cite{ABG}, \cite{MM1} and \cite{MM2}:
they correspond to the case in which the balls $B(x,r)=x+r\,B(0,1)$ in \eqref{average}  are replaced by the family of sets
$x+r\,K$, where $K$ is any fixed reference convex body. 
It is proved that, if $\Omega$ is a $K$-dense set for $\Gamma=\partial\Omega$, $0<|\Omega|<\infty$ and $r_0=\infty$, then both $K$ and $\Omega$ must 
be ellipses  (homothetic to one another) if $N=2$ (\cite{ABG}, \cite{MM1}) and ellipsoids if $N\ge 3$ (\cite{MM2}).

In the present paper, we work in $\mathbb R^3$ and
complete  the classification of unbounded stationary isothermic surfaces initiated in~\cite{MPS2006tams}, by the following theorem and corollary.
 
\begin{theorem}
\label{th:uniformly dense} 
Let $\Omega\subset\mathbb R^3$ be a domain with unbounded and connected boundary $\partial\Omega$ such that $\partial\Omega = \partial\left(\mathbb R^3\setminus \overline{\Omega}\right)$ and let $D$ be a domain with $\overline{D} \subset \Omega$.  Consider a connected component $\Gamma$ of $\partial D$ satisfying 
\begin{equation}
\label{Gamma is the nearest}
\mbox{\rm dist}(\Gamma,\partial\Omega) = \mbox{\rm dist}(\partial D,\partial\Omega)
\end{equation}
and suppose that $D$ 
satisfies the interior cone condition on $\Gamma$.\footnote{This assumption can be replaced by requiring that $\Gamma$ satisfies the following {\it center of mass condition}: for any $x\in\Gamma$, there is a number $r\in (0, r_0)$ such that $x$ is not the center of mass of $\Omega^c\cap B(x,r)$.} 
\par
If $\Omega$ is uniformly dense in $\Gamma$, then $\partial\Omega$ and $\Gamma$ must be  either parallel planes or co-axial circular cylinders.
\end{theorem}
\begin{corollary}
\label{th:stationary isothermic} Let $\Omega$, $D$ and $\Gamma$ be as in {\rm Theorem~\ref{th:uniformly dense}}. Assume that  $\Gamma$ is a stationary isothermic surface of the solution $u$ of the Cauchy problem~\eqref{cauchy}. 
\par
Then $\partial\Omega$ and $\Gamma$ must be  either parallel planes or co-axial circular cylinders.
\end{corollary}

Recall that any point $x \in \Gamma$ has the same distance $R$ from $\partial\Omega$ if $\partial\Omega = \partial\left(\mathbb R^3 \setminus \overline{\Omega}\right)$ and $\Omega$ is uniformly dense  in $\Gamma$.
By observing that 
\begin{eqnarray*}
&&1-\rho(x,r) = \frac {|\Omega^{c}\cap B(x,r)|}{|B(x,r)|}\ \mbox{ for every } r >0,
\\
&&\mbox{ and }\ \frac d{dr}|\Omega^{c}\cap B(x,r)| = |\Omega^{c}\cap\partial B(x,r)|\ \mbox{ for almost every } r >0,
\end{eqnarray*}
where $|\Omega^{c}\cap\partial B(x,r)|$ denotes the $2$-dimensional Hausdorff measure of the set $\Omega^{c}\cap\partial B(x,r)$, we introduce another ratio:
\begin{equation}
\label{surface-measure-in-uniformly-dense}
\sigma(x,r) = \frac {|\Omega^{c}\cap \partial B(x,r)|}{|\partial B(x,r)|}.
\end{equation}
Then, the proof of Theorem \ref{th:uniformly dense} relies on the observation that, being $\Omega$ uniformly dense in $\Gamma$,  $\sigma(x,r)$ and all the coefficients of its asymptotic expansion for $r\to R+0$ are independent of $x$ if $x\in\Gamma$. The regularity in $r$ of $\sigma(x,r)$ descends from that of the uniformly dense domain $\Omega$, that we derive in Section \ref{section2} in general dimension and  under the assumption that  $\partial\Omega = \partial\left(\mathbb R^N\setminus \overline{\Omega}\right)$.
The computation of the first coefficient of the asymptotic expansion of $\sigma(x,r)$ for $r\to R+0$
was already carried out in \cite{MPS2006tams} for any $N$, while that of the second one is performed for $N=3$ in Proposition \ref{prop:asymptotics}, which is the most technical part of the paper (in the Appendix, we collect the calculations of some definite integrals needed in its proof).
\par
A key role is played by Propositions \ref{prop:constant mean curvature} and \ref{prop:asymptotics-2}
that give useful geometrical insight about the first and second coefficients of
the aforementioned asymptotic expansion.
In summary, if $\Omega$ is uniformly dense in $\Gamma$, those propositions imply the existence of an intermediate surface $\Gamma_*$ between $\partial\Omega$ and $\Gamma$, parallel to both $\partial\Omega$ and  $\Gamma$, that has constant mean curvature $H^*$ (Proposition \ref{prop:constant mean curvature}) and of a polynomial $\Psi=\Psi(t)$ of degree $4$ at most such that
\begin{equation}
\label{transnormal}
\Vert\nabla K^*\Vert^2=\,\Psi( K^*)  \ \mbox{ on } \ \Gamma_*
\end{equation}
 (see Proposition \ref{prop:asymptotics-2} and
the arguments yielding \eqref{transnormal Gaussian curvature} in Section \ref{section4}). Here, $\Vert\nabla K^*\Vert$ is the length of the gradient 
(with respect to the induced metric of $\Gamma_*$) of the Gauss curvature $K^*$ of $\Gamma_*$. In particular, \eqref{transnormal} tells us that $K^*$ is a {\it transnormal function}  if  $K^*$ is not constant(see \cite{W1987MathAnn}, \cite{Mi2013DiffGeoAppl}, \cite{B1973QJMathOxford}).
\par
The proof of our classification of uniformly dense sets and stationary isothermic surfaces --- Theorem \ref{th:uniformly dense}  and Corollary \ref{th:stationary isothermic} --- is in Section \ref{section4}. We obtain it in two ways: by combining ideas from the theories of minimal surfaces and surfaces with constant mean curvature properly embedded in $\mathbb R^3$, and the theory of transnormal functions; by directly checking that the Gauss curvature of catenoids, helicoids and unduloids (that, together with planes and circular cylinders, are the only surfaces of constant mean curvature that we need to consider, as we will show) does not satisfy Eq.~\eqref{transnormal} with a polynomial function $\Psi$.
\par
Finally, in Section \ref{section5} we present a generalization of \cite[Theorem 1.4, p. 4824]{MPS2006tams} by using the theory of properly embedded minimal surfaces of finite topology in $\mathbb R^3$. 


\setcounter{equation}{0}
\setcounter{theorem}{0}

\section{Regularity of uniformly dense sets}
\label{section2}

For later use, we introduce some notations and recall some well-known facts.
We define the parallel surface 
$$
\Gamma_\rho=\{x\in\Omega:\mbox{\rm dist}(x,\partial\Omega)=\rho\} \quad \mbox{for} \quad 0<\rho<R.
$$
Also, $\nu$ and  $\kappa_1, \cdots, \kappa_{N-1}$ will denote the {\it inward} unit normal vector to $\partial\Omega$ and the principal curvatures of $\partial\Omega$ with respect to $\nu$ at a point $\xi\in\partial\Omega$. For notational simplicity, the explicit dependence of these quantities on the point $\xi$ will be indicated only when it is needed to avoid ambiguities. However, be aware that $\hat\nu$ and 
$\hat\kappa_1, \cdots, \hat\kappa_{N-1}$ will denote the {\it outward} unit normal vector to $\partial D$  on $\Gamma$ and the principal curvatures of $\Gamma$ with respect to $\hat\nu$ at the point $x=\xi+R\,\nu\in\Gamma$.

In the spirit of \cite[Lemma 3.1]{MS2013mmas} and
by the arguments used in \cite[the proofs of Theorems 2.4 and 2.5]{MPS2006tams}, we obtain the following lemma 
that is partially motivated by Remark \ref{rmk:self-contact} below.

\begin{lemma}
\label{le:regularity} Let $\Omega$ be a domain in $\mathbb R^N$ with $N \ge 2$ and $\partial\Omega = \partial\left(\mathbb R^N\setminus \overline{\Omega}\right)$, and let $D$ be a domain in $\mathbb R^N$ with $\overline{D} \subset \Omega$. Consider a connected component $\Gamma$ of $\partial D$ satisfying \eqref{Gamma is the nearest} and suppose that $D$ satisfies the interior cone condition on $\Gamma$.

If $\Omega$ is uniformly dense in $\Gamma$, then the following properties hold:

\begin{itemize}
\item[\rm (1)] There exists a number $R > 0$ such that $\mbox{\rm dist}(x, \partial\Omega) = R$ for every $x \in \Gamma$; 
\item[\rm (2)] $\Gamma$ is a real analytic hypersurface embedded in $\mathbb R^N$;
\item[\rm (3)] there exists a connected component $\gamma$ of $\partial\Omega$, which is also a real analytic hypersurface embedded in $\mathbb R^N$,
such that the mapping $\gamma \ni \xi \mapsto x(\xi) \equiv \xi + R\nu(\xi) \in \Gamma$ 
is a diffeomorphism; 
in particular, $\gamma$ and $\Gamma$ are parallel hypersurfaces at distance $R$;
\item[\rm (4)] it holds that
\begin{equation}
\label{bounds of curvatures}
\kappa_j< \frac 1R\ \mbox{ on } \gamma,  \mbox{ for every } \ j=1,\dots, N-1;
\end{equation}

\item[\rm (5)] there exists a number $c > 0$ such that
\begin{equation}
\label{monge-ampere}
\prod_{j=1}^{N-1} \left( 1-R\kappa_j\right) = c\ \mbox{ on }\ \gamma.
\end{equation}
\end{itemize}
\end{lemma}

\noindent
\begin{proof} Since $\Omega$ is uniformly dense in $\Gamma$, there exists $r_{0} \in (\mbox{\rm dist}(\Gamma, \partial\Omega), +\infty]$ such that for every fixed $r \in (0, r_{0})$ the function $x \mapsto |\Omega^{c}\cap B(x,r)|$ is also constant on $\Gamma$. Therefore, property~(1) holds for some $R > 0$, since  $\partial\Omega = \partial\left(\mathbb R^N\setminus \overline{\Omega}\right)$ and $\Gamma \cap \partial\Omega = \varnothing$.  Moreover  Eq.~\eqref{Gamma is the nearest} yields that $R =\mbox{\rm dist}(\Gamma,\partial\Omega) = \mbox{\rm dist}(\partial D,\partial\Omega)$. 

First, let us show that $\Gamma$ is a $C^\infty$ hypersurface. Take an arbitrary function $\eta \in C^\infty_0(0, r_0)$ and set $\psi(x) = \eta(|x|)$ for $x \in \mathbb R^N$. Then $\psi \in C_0^\infty(\mathbb R^N)$, $\mbox{\rm supp }(\psi) \subset B(0, r_0)$ and the convolution $\psi \star  {\mathcal X}_{\Omega^c}$ belongs to $C^\infty(\mathbb R^N)$. Moreover we have that
$$
\psi \star  {\mathcal X}_{\Omega^c}(x) = \int_{\Omega^c\cap B(x,r_0)} \eta(|x-y|)\ dy= \int_0^{r_0}\eta(r) |\Omega^c\cap\partial B(x,r)|\ dr,
$$
where $|\Omega^c\cap\partial B(x,r)|$ denotes the $(N-1)$-dimensional Hausdorff measure of the set $\Omega^c\cap\partial B(x,r)$. 

The function $\psi \star  {\mathcal X}_{\Omega^c}$ is constant on $\Gamma$.  In fact, if we fix two points $p, q \in \Gamma$ arbitrarily, being $\Omega$ uniformly dense in $\Gamma$, we have that
$$
|\Omega^{c}\cap B(p,r)| = |\Omega^{c}\cap B(q,r)|\ \mbox{ for every } r \in (0,r_0)
$$
and hence, by differentiating with respect to $r$ both sides, that
\begin{equation}
\label{Hausdorf measure invariant}
|\Omega^c\cap\partial B(p,r)| = |\Omega^c\cap\partial B(q,r)|\ \mbox{ for almost every } r \in (0,r_0).
\end{equation}

Thus, if we show that for every $x \in \Gamma$ there exists a function $\eta \in C^\infty_0(0, r_0)$ such that
$\nabla \left(\psi \star  {\mathcal X}_{\Omega^c}\right)(x) \not = 0$, then we can conclude that $\Gamma$ is a $C^\infty$ hypersurface, by the implicit function theorem. Suppose that there exists a point $x_0 \in \Gamma$ such that
$$
\nabla \left(\psi \star  {\mathcal X}_{\Omega^c}\right)(x_0) = 0\ \mbox{ for every }  \eta \in C^\infty_0(0, r_0);
$$
it follows that
 $$
 \int_0^{r_0} \Biggl[\int_{\Omega^c\cap\partial B(x_0,r)}(x_0 - y)\,dS_y\Biggr]\frac{\eta^\prime(r)}{r}\, dr = 0\ \mbox{ for every }  \eta \in C^\infty_0(0, r_0),
 $$
where $dS_y$ denotes the area element of the sphere $\partial B(x_0,r)$. This, together with the fact that $\Omega^c\cap\partial B(x_0,r) = \varnothing$ for $0 < r < R$, gives that the surface integral in the brackets is zero for almost every $r \in (0, r_0)$,
and hence that
$$
\int_{\Omega^c\cap B(x_0,r)}(x_0 - y)\ dy = 0 \ \mbox{ for every } r \in (0, r_0)
$$
--- that is, $x_0$ must be the center of mass of $\Omega^c\cap B(x_0,r)$ for every $r \in (0, r_0)$. 

By the same argument as in \cite[the proof of Theorem 2.5]{MPS2006tams}, the interior cone condition for $\Gamma$ gives a contradiction. Thus, $\Gamma$ is a $C^\infty$ hypersurface embedded in $\mathbb R^N$.\footnote{It is clear
that the following assumption would suffice:
$\Gamma$ is an immersed topological surface satisfying the  {\it center of mass condition}, as defined in the previous footnote.} 

Now,  since $\Gamma$ is a connected component of $\partial D$, we notice that, in view of  \eqref{Gamma is the nearest}, property~(1) and the smoothness of $\Gamma$ imply that
\begin{equation}
\label{one-to-one}
\mbox{ for each $x \in \Gamma\ $ there exists a unique $\xi \in \partial\Omega$ satisfying $x \in \partial B(\xi,R)$},
\end{equation}
since $\xi - x$ must be parallel to $\hat\nu(x)$.
Note that $\xi = x + R\hat\nu(x)$, and in view of property (1) and  (\ref{one-to-one}), comparing  the principal curvatures at $x$ of $\Gamma$ with those of the sphere $\partial B(\xi,R)$ yields that
\begin{equation}
\label{curvature-weakinequalityGamma}
\hat\kappa_j \le \frac 1R\ \mbox{ on } \Gamma,  \mbox{ for every } \ j=1,\dots, N-1.
\end{equation}
\par
Since $\Gamma$ is a connected component of $\partial D$, then $\Gamma$ is oriented and divides $\mathbb R^N$ into two domains. Let $E$ be the one of them which does not intersect $D$.  By property~(1) and \eqref{Gamma is the nearest},  $E \cap \left(\mathbb R^N \setminus \overline{\Omega}\right)$ contains a point, say,  $z.$ Set  $R_0= \mbox{ dist}(z, \Gamma) $. Then $R_0 > R$ and there exists a point $p_0 \in \Gamma$ such that $R_0 = |z-p_0|.$ Comparing the principal curvatures at $p_0$ of $\Gamma$ with those of the sphere $\partial B(z,R_0),$ yields that $\displaystyle \hat\kappa_j(p_0) \le \frac 1{R_0} < \frac 1R$ for every $j = 1, \dots, N-1$. By continuity, there exists a small $\delta_0 > 0$ such that  
\begin{equation}
\label{curvature-strictinequalityGamma}
\hat\kappa_j(x) < \frac 1R\quad\mbox{ for every }\  x \in \Gamma\cap \overline{B(p_0,\delta_0)}
\ \mbox{ and every } \ j=1,\dots, N-1,
\end{equation}
and the mapping $\Gamma\cap B(p_0,\delta_0) \ni x \mapsto \xi(x) \equiv x + R\,\hat\nu(x) \in \partial\Omega$ is a diffeomorphism onto its image $\gamma_{0}$ given by
$$
\gamma_{0}= \xi\left(\Gamma\cap B(p_0,\delta_0)\right) \left(\subset \partial\Omega \right).
$$
Hence $\gamma_{0}$ is a portion of a $C^\infty$  hypersurface, since $\Gamma$ is a $C^\infty$ hypersurface. 
\par
Notice that
the principal curvatures $\kappa_1, \cdots, \kappa_{N-1}$ of $\gamma_{0}$ satisfy
$$
-\kappa_j(\xi(x))= \frac {\hat\kappa_j(x)}{1-R\hat\kappa_j(x)} \ \mbox{ for every }  
x \in \Gamma\cap B(p_0,\delta_0)\mbox{ and every }\ j = 1, \dots, N-1.
$$
Therefore, since $1-R\kappa_j(\xi(x)) = 1/(1-R\hat\kappa_j(x))$, we see that (\ref{curvature-strictinequalityGamma}) is equivalent to
\begin{equation}
\label{curvature-strictinequalityOmega}
\kappa_j < \frac 1R \ \mbox{ on }\  \gamma_{0}  \ \mbox{ for every } \ j=1,\dots, N-1.
\end{equation}

Here, notice that $\partial\Omega$ may have a point of selfcontact, since we only assume that $\partial\Omega = \partial\left(\mathbb R^N\setminus \overline{\Omega}\right)$. For this reason set
$$
\gamma_{0}^{*} = \{ \xi \in \gamma_{0}\ :\ \xi \mbox{ is a point of selfcontact of } \partial\Omega \}.
$$
Then $\gamma_{0}^{*}$ does not contain any interior points in $\gamma_{0}$, 
since  $\gamma_{0}$ is a portion of a $C^\infty$  hypersurface and $\partial\Omega = \partial\left(\mathbb R^N\setminus \overline{\Omega}\right)$. 

Let $P, Q \in \gamma_{0}\setminus \gamma_{0}^{*}$ be any two points and set $\xi(p) = P$, $\xi(q) = Q$ for $p, q \in 
\Gamma\cap B(p_0,\delta_0)$. Then, it follows from \eqref{Hausdorf measure invariant} and  
 the smoothness of $\gamma_{0}$ that there exists a small number $\varepsilon > 0$  satisfying
$$
|\Omega^c\cap\partial B(p,r)| = |\Omega^c\cap\partial B(q,r)|\ \mbox{ for  every } r \in (R, R+\varepsilon). 
$$
Hence we can use \cite[Theorem 5.5]{MPS2006tams}, with $\mathbb R^N \setminus \overline{\Omega}$ in place of $\Omega$, to get
\begin{equation}
\label{conclusion by asymptotics}
\left[\prod_{j=1}^{N-1}\left(1-R\kappa_j(P)\right)\right]^{-\frac12} = \left[\prod_{j=1}^{N-1}\left(1-R\kappa_j(Q)\right)\right]^{-\frac12}.
\end{equation}
Therefore, since $\gamma_{0}^{*}$ does not contain any interior points in $\gamma_{0}$, by continuity we conclude that
\begin{equation}
\label{monge-ampere-formula-local}
\prod_{j=1}^{N-1}\left(1-R\,\kappa_j\right) = c\ \mbox{ on }\ \gamma_{0}
\end{equation}
where, for instance, $c$ is the (positive) value of the right-hand side of \eqref{monge-ampere-formula-local} at the point $P_0=\xi(p_0) \in \gamma_{0}$.
Since $1-R\kappa_j(\xi(x)) = 1/(1-R\hat\kappa_j(x))$, we see that
$$
\prod_{j=1}^{N-1}\left[1-R\,\hat\kappa_j(x)\right]= c^{-1}\ \mbox{ for every }\ x \in \Gamma\cap B(p_0,\delta_0).
$$

Define a set $J \subset \Gamma$ by
$$
J = \Biggl\{ p \in \Gamma : \max_{1\le j \le N-1} \hat\kappa_j(p) < \frac 1R \mbox{ and } \prod_{j=1}^{N-1}\left[1-R\,\hat\kappa_j(p)\right] = c^{-1} \Biggr\}.
$$
By the previous argument we notice that $J$ is a relatively open subset of $\Gamma$ and $J \not = \varnothing$.  Moreover, $J$ is a relatively closed subset of $\Gamma$.
Indeed, for any sequence of points $p_k\in  J$ converging to some $p \in \Gamma$ as $k \to \infty$, 
in the limit we would get that
\begin{equation}
\label{monge-ampere type eq}
\max_{1\le j \le N-1} \hat\kappa_j(p) \leq \frac 1R\ \mbox{ and } \prod_{j=1}^{N-1}\left[1-R\,\hat\kappa_j(p)\right] = c^{-1}  (> 0),
\end{equation}
and the second equality implies that the first inequality must be strict; thus, $J$ is closed. Since $\Gamma$ is connected, we conclude that $J = \Gamma$.  Also, the regularity theory for nonlinear elliptic equations implies that 
$\Gamma$ is a real analytic hypersurface, since $\Gamma$ is locally a graph of a function which satisfies a Monge-Amp\`ere type equation coming from the second equality of \eqref{monge-ampere type eq}. 
Let us set
\begin{equation}
\label{definition of gamma}
\gamma = \{ \xi(x) \in \mathbb R^N : x \in \Gamma \}.
\end{equation}
Then $\gamma$ does not have any points of selfcontact, that is, the mapping $\xi: \Gamma \to \partial\Omega$ is injective. Indeed, suppose that there exists a point of selfcontact $P_{*} \in \gamma$, that is, there exist two open portions $\gamma_{+}, \gamma_{-}$ of the manifold $\gamma$ containing a common point $P_{*}$. Hence we have two points $p_{+}, p_{-} \in \Gamma$ and two inward normal vectors $\nu^{+}$ and $\nu^{-}$ at $P_{*} \in \gamma \left(\subset \partial\Omega\right)$ satisfying
\begin{equation}
\label{opposite position}
\nu^{+}+ \nu^{-}= 0,\ p_{+}= P_{*} + R\,\nu^{+},\ \mbox{ and }\ p_{-}= P_{*} + R\,\nu^{-}.
\end{equation}
Denote by $\kappa^{\pm}_1, \cdots, \kappa^{\pm}_{N-1}$ the principal curvatures of $\gamma_{\pm}$  at $P_{*} \in \gamma$ with respect to the inward unit normal vectors $\nu^{\pm}$ to $\partial\Omega$, respectively.  Then  we observe that 
\begin{equation}
\label{strict inequalities from both sides}
-\frac 1R < \kappa^{\pm}_j < \frac 1R  \ \mbox{ for every } \ j=1,\dots, N-1.\,
\end{equation}
since $J = \Gamma$ and $P_{*}$ is  the point of selfcontact of $\gamma$. 
As before, take a point $Q \in \gamma$  which is not a point of selfcontact and set $\xi(q) = Q$ for $q \in \Gamma$.
Then it follows from \eqref{Hausdorf measure invariant} and  
 the smoothness of $\gamma$ that there exists a small number $\varepsilon_{0} > 0$  satisfying
$$
|\Omega^c\cap\partial B(p_{+},r)| = |\Omega^c\cap\partial B(q,r)|\ \mbox{ for  every } r \in (R, R+\varepsilon_{0}). 
$$
Hence, in view of \eqref{opposite position} and \eqref{strict inequalities from both sides}, we can use \cite[Theorem 5.5]{MPS2006tams} again to get
\begin{equation}
\label{additional monge ampere}
\Biggl[\prod_{j=1}^{N-1}\left(1-R\,\kappa^{+}_j\right)\Biggr]^{-\frac12} - \Biggl[\prod_{j=1}^{N-1}\left(1+R\,\kappa^{-}_j\right)\Biggr]^{-\frac12}= \Biggl[\prod_{j=1}^{N-1}\left(1-R\,\kappa_j(Q)\right)\Biggr]^{-\frac12}.
\end{equation}
This is a contradiction, since \eqref{monge-ampere-formula-local} in which $\gamma_0$ is replaced by $\gamma$ holds true from the fact that $J = \Gamma$.
\par
Therefore, since $\hat\kappa_j< 1/R$ on $\Gamma$ for every $j=1,\dots, N-1$, we see that the injective mapping $\Gamma \ni x \mapsto \xi(x) \equiv x + R\,\hat\nu(x) \in \gamma$ is a real analytic diffeomorphism because of the analyticity of $\Gamma$,  and $\gamma$ is a real analytic hypersurface embedded in $\mathbb R^N$ which is a connected component of
$\partial\Omega$. Since the mapping: $\gamma \ni \xi \mapsto x(\xi) \equiv \xi + R\,\nu(\xi) \in \Gamma$ is the inverse mapping of the previous diffeomorphism, property~(3) holds.  Both properties~(4) and~(5) follow from the fact that $J = \Gamma$. 
The proof is complete.
\end{proof}


\begin{remark} 
\label{rmk:self-contact}
{\rm In \cite[Lemma 3.1 and its proof, pp. 2026--2029]{MS2013mmas}, the first and third authors of this paper did not take care of the case in which 
$\gamma$ has points of self-contact. 
Thus, Lemma \ref{le:regularity} completes the proof of {\rm\cite[Lemma 3.1, p. 2026]{MS2013mmas}} for the case of the Cauchy problem. 
\par
Still, the argument we used to obtain \eqref{additional monge ampere} does not work 
in the case of the initial-boundary value problem for the heat equation with boundary value $1$ and initial value $0$ --- the {\it matzoh ball soup} setting considered initially in \cite{MS2002}.
Hence, statement {\rm $3$} of {\rm\cite[Lemma 3.1, p. 2026]{MS2013mmas}} should be corrected in such a way that $\gamma$ is an {\it immersed} hypersurface in $\mathbb R^{N}$. Then  $\gamma$ may have points of self-contact. 
\par
On the contrary, the reflection argument due to Alexandrov works for a bounded domain $\Omega$,
even if $\partial\Omega$ contains points of self-contact  (see \cite{Alex1962annali}). So the statement of Remark right after \cite[Lemma 3.1, p. 2026]{MS2013mmas} still holds true.
}
\end{remark}

The following proposition follows directly from Lemma {\rm \ref{le:regularity}} 
and is one of the key ingredients in the proof of Theorem \ref{th:uniformly dense}. We preliminarily notice that,
under the assumptions of Theorem \ref{th:uniformly dense}, there exists $R > 0$ such that 
\begin{equation}
\label{constant distance to the boundary}
\mbox{ dist}(x, \partial\Omega) = R \ \mbox{ for every } x \in \Gamma,
\end{equation}
since $\partial\Omega = \partial\left(\mathbb R^3 \setminus \overline{\Omega}\right)$ and $\Omega$ is uniformly dense in $\Gamma$.
Also,  since $\partial\Omega$ is connected, 
Lemma {\rm \ref{le:regularity}} and \eqref{constant distance to the boundary}  imply that
$$
\gamma = \partial\Omega,\ \mbox{ and } \mbox{ $\partial\Omega$ and $\Gamma$ are parallel surfaces at distance $R > 0$.}
$$
Furthermore, both $\partial\Omega$ and $\Gamma$ are embedded in $\mathbb R^3$.

\begin{proposition}
\label{prop:constant mean curvature} Under the assumptions of {\rm Theorem  \ref{th:uniformly dense}}, set $\rho_* = R/(1+\sqrt{c})$, where $c > 0$ is the number in \eqref{monge-ampere} in {\rm Lemma \ref{le:regularity}}, and
$$
\Gamma_*= \{ x \in \Omega\ :\  \mbox{\rm dist}(x,\partial\Omega) = \rho_* \}.
$$

Then, $\Gamma_*$ is a real analytic hypersurface parallel to $\partial\Omega$,  $\Gamma_*$ is embedded in $\mathbb R^3$, and  $\Gamma_*$ has a constant mean curvature 
$$
H^* = \frac {1-c}{2R\sqrt{c}},
$$ 
where the normal to $\Gamma_*$ is chosen to point in the same direction as the inward normal to $\partial\Omega$.  In particular, $\Gamma_*$ is a properly embedded surface with
constant mean curvature (or a properly embedded minimal surface when $c=1$) in $\mathbb R^3$, and hence it is complete.
\end{proposition}

\noindent
\begin{proof} Since both $\partial\Omega$ and $\Gamma$ are embedded in $\mathbb R^3$ and the mapping 
$\partial\Omega \ni \xi \mapsto \xi + \rho_*\nu(\xi) \in \Gamma_*$ is a diffeomorphism because  $0 < \rho_* < R$, we see that $\Gamma_*$ is also a real analytic hypersurface embedded in $\mathbb R^3$ and  $\Gamma_*$ is  parallel to both $\partial\Omega$ and $\Gamma$.

For $0 < \rho < R$,  we denote by  $\kappa^\rho_1$ and $\kappa^\rho_2$ the principal curvatures of $\Gamma_\rho$ at $x = \xi + \rho\,\nu(\xi)\in \Gamma_\rho$ with respect to  the unit normal vector to $\Gamma_*$ with the same direction as $\nu(\xi)$. Then
\begin{equation}
\label{the relation of principal curvatures}
\kappa_j(\xi)= \frac {\kappa^\rho_j(x)}{1+\rho\kappa^\rho_j(x)}\ \ (j = 1, 2)\  \mbox{ for every } \xi \in \partial\Omega.
\end{equation}
Substituting these in Eq.~\eqref{monge-ampere} yields
$$
c(1+\rho\kappa^\rho_1)(1+\rho\kappa^\rho_2) = [1+(\rho-R)\kappa^\rho_1][1+(\rho-R)\kappa^\rho_2].
$$
Hence, by letting $\rho = \rho_*$, with $\rho_*=R/(1+\sqrt{c})$, and $\Gamma_* = \Gamma_{\rho_*}$,  we see that 
$$
H^*= \frac {\kappa^{\rho_*}_1+\kappa^{\rho_*}_2}2= \frac {1-c}{2R\sqrt{c}}. 
$$ 

Let us see that $\Gamma_*$ is  {\it properly} embedded in $\mathbb R^3$.  Observe that $\Gamma_* =  \{ x \in \overline{\Omega}\ :\  \mbox{\rm dist}(x,\partial\Omega) = \rho_* \}$ where
$\Omega$ is replaced by $\overline{\Omega}$.  Since the distance function $\mbox{\rm dist}(x,\partial\Omega)$ is continuous on $\mathbb R^3$, we see that $\Gamma_*$ is closed in $\mathbb R^3$.
Let $K$ be an arbitrary compact subset of $\mathbb R^3$. Then $\Gamma_*\cap K$ is also compact in $\mathbb R^3$. Let $\{ p_n\}$ be an arbitrary sequence in $\Gamma_*\cap K$.
By the Bolzano -Weierstra\ss \  theorem,  $\{ p_n\}$ has a convergent subsequence in $\mathbb R^3$. Let $p\in\Gamma_*\cap K$ be its limit point. Since $ p \in \Gamma_*$, the smoothness of $\Gamma_*$  yields that there exists $\delta > 0$ such that $B(p,\delta) \cap \Gamma_*$ is represented by a real analytic graph over the tangent plane of $\Gamma_*$ at $p$. This shows that the above subsequence also converges to $p$ with respect to the induced metric of $\Gamma_*$, which means that $\Gamma_*$ is  properly embedded in $\mathbb R^3$. (Similarly, both $\partial\Omega$ and $\Gamma$ are properly embedded in $\mathbb R^3$.) 
\end{proof}


\setcounter{equation}{0}
\setcounter{theorem}{0}

\section{Asymptotic expansions for $\sigma(x,r)$}
\label{section3}

The second key ingredient in the proof of Theorem~\ref{th:uniformly dense} is Proposition \ref{prop:asymptotics}  below, in which we prove an asymptotic formula for $4\pi\cdot\sigma(x, R+s) (= |\partial B(x, R+s) \cap \Omega^c|/(R+s)^2) $ as $s\to +0$, where $R > 0$ is given in \eqref{constant distance to the boundary} and $\sigma(x,r)$ is defined in \eqref{surface-measure-in-uniformly-dense}. 
The ensuing Proposition \ref{prop:asymptotics-2} will then clarify the geometric meaning of the function $g$ in \eqref{formula-asymptotics}. 
In Proposition \ref{prop:asymptotics}, we choose a principal coordinate system $z=(z_1,z_2,z_3)$ with the origin at $\xi\in\partial\Omega$ and such that, in some neighborhood of $\xi$, $ \partial\Omega$ is represented by the graph $z_3 = \varphi(z_1,z_2)$, with the $z_3$ coordinate axis lying in the direction $-\nu(\xi)$ and
$$
\varphi(z_1,z_2) = -\frac 12 \kappa_1(\xi) z_1^2 -\frac 12 \kappa_2(\xi) z_2^2 + O\left((z_1^2+z_2^2)^{\frac 32}\right) \ \mbox{ as } \sqrt{z_1^2+z_2^2} \to 0.
$$
Hereafter, we abbreviate the partial derivatives of $\varphi$ with respect to $z_1$ and $z_2$ by subscripts:
$$
\varphi_1= \frac {\partial \varphi}{\partial z_1},\ \varphi_{11}= \frac {\partial^2\varphi}{\partial z_1^2},\   \varphi_{112} = \frac {\partial^3\varphi}{\partial z_2\partial z_1^2}\mbox{ and so on.}
$$


\begin{proposition}
\label{prop:asymptotics} 
Let $\xi \in \partial\Omega$ and set $x = \xi + R\nu(\xi) \in \Gamma$. Under the assumptions of {\rm Theorem  \ref{th:uniformly dense}}, we have:
\begin{equation}
\label{formula-asymptotics}
\frac {|\partial B(x, R+s) \cap \Omega^c|}{(R+s)^2} = \frac {2\pi}{\sqrt{c}} \,\frac{s}{R} + \frac \pi{8\, c \sqrt{c}}\,[h(K) + g]\,\left(\frac{s}{R}\right)^2 + O\left(s^{\frac 52}\right)\ \mbox{ as } s \downarrow 0.
\end{equation}
Here, $K = \kappa_1(\xi)\kappa_2(\xi)$ is the Gauss curvature of the surface $\partial\Omega$ at the point $\xi$ and $h$ is a $2$-degree polynomial:
\begin{equation}
\label{defh}
h(t) = (R^2t+c-1)^2 -4 c\, (c+3).
\end{equation}
\par
Moreover, $g\le 0$ on $\partial\Omega$ and $g=0$ if and only if the third-order derivatives $\varphi_{111}$ and $\varphi_{222}$ of the function $\varphi$ defined above vanish at the origin. 
\end{proposition}

\vskip.3cm

The starting point of the proof of this proposition is Lemma \ref{le:Puiseux expansion}, for which we need ad hoc notations and settings, in the spirit of those introduced in \cite{MPS2006tams}.

In fact, we shall use the principal coordinate directions introduced before the statement of Proposition~\ref{prop:asymptotics} without further mention. Also, for sufficiently small $s > 0$, each point $w \in \partial B(x, R+s) \cap \Omega^c$ can be parameterized by a spherical coordinate system with the origin at $x \in \Gamma$ as
$$
w = x + (R+s)(\sin \eta\cos\theta, \sin\eta\sin\theta,\cos\eta), \ 0 \le \eta \le \eta(s,\theta),\ 0 \le \theta < 2\pi,
$$
where $\eta=\eta(s,\theta)\ (0 \le \theta \le 2\pi)$ represents the closed curve  $\partial B(x, R+s) \cap \partial\Omega$ in that system. Notice that, for sufficiently small $s > 0$, $\eta = \eta(s,\theta)$ satisfies
\begin{equation}
\label{spherical-Cartesian}
G(\eta, s, \theta) = 0\ \mbox{ for every }  0 \le \theta \le 2\pi,
\end{equation}
where the function $G=G(\eta,s,\theta)$ is given by
\begin{equation}
\label{auxiliary function for IFT}
G(\eta,s,\theta) = (R+s)\cos\eta - R - \varphi((R+s)\sin\eta\cos\theta,(R+s)\sin\eta\sin\theta). 
\end{equation}
Thus, we obtain:
\begin{equation}
\label{area wanted}
\frac  {|\partial B(x, R+s) \cap \Omega^c|}{(R+s)^2} = \int_0^{2\pi}d\theta\int_0^{\eta(s,\theta)}\sin\eta \ d\eta
=  \int_0^{2\pi}(1-\cos\eta(s,\theta)) d\theta.
\end{equation}

\vskip.3cm


\begin{lemma}
\label{le:Puiseux expansion}
There exists a sequence $\{b_j(\theta)\}_{j=1}^\infty$ such that $b_1>0$ and $\eta=\eta(s,\theta)$ is expanded as the Puiseux series in $s$:
\begin{equation}
\label{Puiseux series in s}
\eta(s,\theta) = \sum_{j=1}^\infty b_j(\theta)s^{\frac j2}\ \mbox{ for small } s \ge 0,
\end{equation}
and as $s \downarrow 0$
 \begin{eqnarray}
 \frac {|\partial B(x, R+s) \cap \Omega^c|}{(R+s)^2} &=&  \frac 12 \int_0^{2\pi}b_1^2d\theta\ s +  \int_0^{2\pi}b_1b_2d\theta\ s^{\frac 32} \nonumber
 \\
 &\quad& + \frac 12 \int_0^{2\pi}\left(b_2^2+2b_1b_3-\frac1{12}b_1^4\right)d\theta\ s^2 + O\left(s^{\frac52}\right).\label{asymptotics with b_j}
 \end{eqnarray}
 \end{lemma}
 
 \noindent
\begin{proof} Since the function $G$ given by \eqref{auxiliary function for IFT} satisfies
$$
G(0,0,\theta) = 0\ \mbox{ and }\ \frac {\partial G}{\partial s}(0,0,\theta) = 1\ \mbox{ for every }  0 \le \theta \le 2\pi,
$$
by the implicit function theorem there exists a sequence $\{a_j(\theta)\}_{j=1}^\infty$ such that $s = s(\eta, \theta)$ is written as
\begin{equation}
\label{power series in eta}
s = \sum_{j=1}^\infty a_j(\theta)\eta^j\ \mbox{ for small } \eta \ge 0.
\end{equation}
By differentiating the identity $G(\eta, s(\eta, \theta), \theta) = 0$ with respect to $\eta$, we get
\begin{eqnarray*}
0 =&& s_\eta\cos\eta+ (R+s)(-\sin\eta)
\\
&& - \varphi_{1}((R+s)\sin\eta\cos\theta,(R+s)\sin\eta\sin\theta)\left( (R+s)\cos\eta\cos\theta + s_\eta\sin\eta\cos\theta\right)
\\
&& - \varphi_{2}((R+s)\sin\eta\cos\theta,(R+s)\sin\eta\sin\theta)\left( (R+s)\cos\eta\sin\theta + s_\eta\sin\eta\sin\theta\right).
 \end{eqnarray*}
 By setting $\eta=0$, we get 
 \begin{equation}
 \label{coefficient a1}
 a_1(\theta) = s_\eta(0,\theta) = 0. 
 \end{equation}
 Differentiating the above identity with respect to $\eta$ once  more and putting $\eta=0$ yield that
 $$
 0 = s_{\eta\eta}(0,\theta) -R + R^2\kappa_1(\xi)\cos^2\theta + R^2\kappa_2(\xi)\sin^2\theta, 
 $$
and hence 
\begin{eqnarray}
a_2(\theta) &=& \frac12 s_{\eta\eta}(0,\theta) = \frac 12 R\left[ (1-R\kappa_1(\xi))\cos^2\theta + (1-R\kappa_2(\xi))\sin^2\theta \right] \nonumber
\\
&\ge& \frac 12 R\left[1-R\max\{\kappa_1(\xi), \kappa_1(\xi)\}\right] > 0.\label{coefficient a2}
\end{eqnarray}
In view of \eqref{power series in eta}, \eqref{coefficient a1}, and  \eqref{coefficient a2}, we see that there exists a sequence $\{b_j(\theta)\}_{j=1}^\infty$ such that $b_1>0$ and $\eta=\eta(s,\theta)$ is expanded as the Puiseux series \eqref{Puiseux series in s} in $s$.
With the aid of \eqref{Puiseux series in s}, we calculate for $\eta = \eta(s,\theta)$ 
\begin{eqnarray*}
 1-\cos\eta &=& \frac 12\eta^2-\frac 1{24}\eta^4 + O\left(\eta^{6}\right)
 \\
 &=& \frac 12 b_1^2 s + b_1b_2s^{\frac 32} + \frac 12\left(b_2^2+2b_1b_3-\frac1{12}b_1^4\right)s^2 + O\left(s^{\frac52}\right)\ \mbox{ as } s \downarrow 0,
 \end{eqnarray*}
 and hence \eqref{area wanted} implies \eqref{asymptotics with b_j}, as we desired to prove. 
 \end{proof}

\vskip.3cm

\begin{proofC}
Since $\partial\Omega$ is a real analytic hypersurface by Lemma {\rm \ref{le:regularity}},  we can write
\begin{equation}
\label{Taylor expansion for varphi}
\varphi(z_1,z_2) = \sum_{k=2}^\infty P_k(z_1,z_2)\ \mbox{ for sufficiently small } \sqrt{z_1^2+z_2^2}\,,
\end{equation}
where each $P_k(z_1,z_2)$ is a homogeneous polynomial of degree $k$ and in particular 
\begin{equation}
\label{quadratic term}
P_2(z_1,z_2) = -\frac 12 (\kappa_1(\xi) z_1^2 +\kappa_2(\xi) z_2^2). 
\end{equation}

Now we compute the integrands of the expansion~\eqref{asymptotics with b_j}. For $P_k$ given in \eqref{Taylor expansion for varphi},  we write
 $$
 P_k(v) = P_k(\cos\theta,\sin\theta)\ \mbox{ for } v = (\cos\theta,\sin\theta).
 $$
By substituting this and \eqref{Taylor expansion for varphi} into \eqref{spherical-Cartesian}, since
$$
\cos\eta = 1-\frac 12\eta^2+\frac 1{24}\eta^4 + O\left(\eta^{6}\right) \ \mbox{ and } \ \sin\eta = \eta - \frac 16\eta^3+ O\left(\eta^{5}\right),
$$
we see that
$$
(R+s)\left(1-\frac 12\eta^2+\frac 1{24}\eta^4 + O\left(\eta^{6}\right)\right) - R - \sum_{k=2}^\infty (R+s)^k\left(\eta - \frac 16\eta^3+ O\left(\eta^{5}\right)\right)^k P_k(v) = 0.
$$
Then, with \eqref{Puiseux series in s} ($\eta = b_1s^\frac 12+ b_2 s + b_3s^{\frac 32}+ O\left(s^{2}\right)$) in hand, we equate to zero the coefficients of $s, s^{\frac 32},$ and $s^2$. The coefficient of $s$ gives
\begin{equation}
\label{coefficient in s1}
1-\frac 12 Rb_1^2 - R^2b_1^2P_2(v) = 0.
\end{equation}
The coefficient of $s^{\frac32}$ gives
\begin{equation}
\label{coefficient in s32}
-Rb_1b_2-2R^2b_1b_2P_2(v)-R^3b_1^3P_3(v)= 0.
\end{equation}
The coefficient of $s^{2}$ gives
\begin{eqnarray}
&&-\frac12\left\{ b_1^2+Rb_2^2 +2Rb_1b_3\right\} + \frac 1{24}R b_1^4 \nonumber
\\
&& -P_2(v)\left\{ 2Rb_1^2 +R^2b_2^2 + 2R^2b_1b_3 -\frac 13 R^2b_1^4 \right\} \nonumber
\\
&&-3R^3b_1^2b_2P_3(v) - R^4b_1^4P_4(v)= 0.\label{coefficient in s2}
\end{eqnarray}
\par
Now, set 
\begin{equation}
\label{sigma_j}
\sigma_j= 1-R\kappa_j(\xi) > 0\ \mbox{ for } j =1, 2.
\end{equation}
Notice that 
\begin{equation}
\label{monge_ampere constant c}
\sigma_1\sigma_2 = c,
\end{equation}
where $c > 0$ is the positive number given by \eqref{monge-ampere}.

In view of \eqref{asymptotics with b_j}, with the aid of \eqref{coefficient in s1} and \eqref{quadratic term}, we obtain:
\begin{equation}
\label{main term}
 \frac 12\, b_1^2 = \frac 1{R(1+2RP_2(v))} = \frac 1{R(\sigma_1\cos^2\theta + \sigma_2\sin^2\theta)}.
 \end{equation}
The coefficient of $s$ in \eqref{formula-asymptotics} is thus easily computed from this formula, by using 
\eqref{monge_ampere constant c}  and \eqref{starting point}:
\begin{equation}
\label{the first coefficient}
\frac 12 \int_0^{2\pi} b_1^2\, d\theta =  \frac {2\pi}{R\sqrt{c}}\,;
\end{equation}
here $c $ is the positive number given by \eqref{monge-ampere}.

By using~\eqref{coefficient in s32} and \eqref{main term}, we have:
\begin{equation}
\label{2nd term}
b_1b_2= -\frac {2^{\frac32}R^{\frac12} P_3(v)}{(\sigma_1\cos^2\theta + \sigma_2\sin^2\theta)^{\frac 52}},
\end{equation}
and hence we get  
\begin{equation}
\label{the second coefficient}
\int_0^{2\pi} b_1b_2 d\,\theta = 0,
\end{equation}
since $b_1b_2$ is the sum of odd functions of either $\cos\theta$ or $\sin\theta$ because of \eqref{2nd term}.
Thus, by \eqref{asymptotics with b_j}
the coefficient of $s^{3/2}$ in \eqref{formula-asymptotics} is zero.

Finally, it follows from \eqref{coefficient in s2}, \eqref{main term}, and \eqref{2nd term} that
\begin{eqnarray}
\frac12\left(b_2^2+2b_1b_3-\frac1{12}b_1^4\right) &=& -\frac 7{6R^2(\sigma_1\cos^2\theta + \sigma_2\sin^2\theta)^2} + \frac 1{6R^2(\sigma_1\cos^2\theta + \sigma_2\sin^2\theta)^3}\nonumber
\\
&&-  \frac {4P_2(v)}{R(\sigma_1\cos^2\theta + \sigma_2\sin^2\theta)^2} + \frac {4P_2(v)}{3R(\sigma_1\cos^2\theta + \sigma_2\sin^2\theta)^3}\nonumber
\\
&&+ \frac {12R^2(P_3(v))^2}{(\sigma_1\cos^2\theta + \sigma_2\sin^2\theta)^4} - \frac {4RP_4(v)}{(\sigma_1\cos^2\theta + \sigma_2\sin^2\theta)^3}.\label{3rd coefficient}
\end{eqnarray}

A long  but important computation, that is carried out in the Appendix (Lemmas \ref{le: the first four terms in 3rd in terms of c and K} and \ref{le: the sum of the last two terms in 3rd in terms of c and K}), then yields 
the coefficient of $s^2$ in \eqref{formula-asymptotics}, that is
\begin{multline*}
 \frac {R^2}{2} \int_0^{2\pi}\left(b_2^2+2b_1b_3-\frac1{12}b_1^4\right)d\theta =\frac{\pi}{8 c\sqrt{c}} \biggl\{ (R^2 K+c-1)^2 -4c\,(c+3) \\
 - \frac 43\,c\,R^4\left[(1-R \kappa_1)^{-3}\left(\varphi_{111}\right)^2+ (1-R \kappa_2)^{-3}\left(\varphi_{222}\right)^2\right]\biggr\},
 \end{multline*}
where the derivatives of $\varphi$ are evaluated at $(0,0)$.

In the last formula, we set 
\begin{equation}
\label{definition of the function g}
g = - \frac 43\,c\,R^4\left[(1-R \kappa_1)^{-3}\left(\varphi_{111}\right)^2+ (1-R \kappa_2)^{-3}\left(\varphi_{222}\right)^2\right],
\end{equation}
that is clearly non-positive and is null if and only if both third derivatives vanish. 
\end{proofC}

\vskip.3cm

In the next proposition, $\kappa^{*}_{1}$ and $\kappa^{*}_{2}$ denote the principal curvatures of $\Gamma_{*}$ at the point $x^{*} = \xi + \rho_{*}\nu(\xi)$ defined by 
\eqref{the relation of principal curvatures} with
$\rho_{*}= R/(1+\sqrt{c})$, $K^{*}= \kappa^{*}_{1}\,\kappa^{*}_{2}$, and $g$ is the function
appearing in \eqref{formula-asymptotics}, whose expression is given by \eqref{definition of the function g}.


\begin{proposition}
\label{prop:asymptotics-2} It holds that
\begin{equation}
\label{formula-gradient_of_Gaussian_curvature}
(\kappa^{*}_{2}-\kappa^{*}_{1})^{2} g = -\frac{4R^{4}}{3\sqrt{c}}\cdot\frac{\|\nabla K^{*}\|^{2}}{(1+\rho_{*}\kappa^{*}_{1})^{3}(1+\rho_{*}\kappa^{*}_{2})^{3}}
\end{equation}
or, in terms of the invariants $H^*$ and $K^*$,
\begin{equation}
\label{formula-gradient_of_Gaussian_curvature2}
\|\nabla K^{*}\|^{2}=\frac{3\sqrt{c}}{R^{4}}\,g\,[K^*-(H^*)^2]\,[1+2\,\rho_{*}H^{*}+\rho_*^2 K^*]^{3},
\end{equation}
where $\|\nabla K^{*}\|$ is the length of the gradient of the Gauss curvature $K^{*}$  with respect to the induced metric of the hypersurface $\Gamma_{*}$.
\end{proposition}

\begin{proof}
Note that
$$
 \kappa_{j}= \frac {\kappa^{*}_{j}}{1+\rho_{*}\kappa^{*}_{j}} \ \mbox{  for } j = 1, 2, \quad \kappa^{*}_{1}+\kappa^{*}_{2} = \frac {1-c}{R\sqrt{c}} \  \mbox{ and }\ \rho_{*} = \frac R{1+\sqrt{c}}.
 $$

First notice that formula \eqref{formula-gradient_of_Gaussian_curvature} holds true
if $\kappa_1= \kappa_2$ or, which is equivalent, if $\kappa^{*}_{1}=\kappa^{*}_{2}$.
In fact,  in this case, the Gauss curvature $K^*$ of $\Gamma_*$ attains its maximum value $(H^*)^2$ (at $x^*$). This means that $\nabla K^*$ vanishes (at $x^*$)
and hence both sides of \eqref{formula-gradient_of_Gaussian_curvature} equal zero.


We now suppose that $\kappa_1\not=\kappa_2$. Thus, by using the Monge principal coordinate system (\cite[p. 156]{BL1973springer-book}),  we have as $\sqrt{z_{1}^{2}+z_{2}^{2}} \to 0$ that
\begin{eqnarray*}
\varphi(z_{1},z_{2}) &=& -\frac12\kappa_{1}z_{1}^{2}-\frac12\kappa_{2}z_{2}^{2}
\\
&& -\frac 16\left\{\frac{\partial \kappa_{1}}{\partial z_{1}}z_{1}^{3}+3\frac{\partial \kappa_{1}}{\partial z_{2}}z_{1}^{2}z_{2} + 3\frac{\partial \kappa_{2}}{\partial z_{1}}z_{1}z_{2}^{2} + \frac{\partial \kappa_{2}}{\partial z_{2}}z_{2}^{3}\right\} + O\left((z_{1}^{2}+z_{2}^{2})^{2}\right).
\end{eqnarray*}
Therefore, we have at $(0,0)$:
$$
\varphi_{111}= - \frac{\partial \kappa_{1}}{\partial z_{1}}\ \mbox{ and }\ \varphi_{222} = -\frac{\partial \kappa_{2}}{\partial z_{2}}.
$$
Hence, we obtain from \eqref{definition of the function g}:
\begin{equation}
\label{pre-version} 
-\frac{3\sqrt{c}\, g}{4R^4}\,  (1+\rho_*\kappa_1^*)^{3}(1+\rho_*\kappa_2^*)^{3}
=
(1+\rho_{*}\kappa^{*}_{1})^{2}\left(\frac{\partial \kappa^{*}_{1}}{\partial z_{1}}\right)^{2} + (1+\rho_{*}\kappa^{*}_{2})^{2}\left(\frac{\partial \kappa^{*}_{1}}{\partial z_{2}}\right)^{2}.
\end{equation}
\par
By recalling that $\Gamma^{*}$ is parameterized in $z$ by
$$
(z, \varphi(z)) -\rho_{*}\frac 1{\sqrt{1+|\nabla_z\varphi(z)|^{2}}}\left(-\nabla_{z}\varphi(z),1\right),
$$
 we have at $z=0$ (that is at $x^{*} \in \Gamma_{*}$) that
\begin{eqnarray*}
\|\nabla K^{*}\|^{2} &=& (1-\rho_{*}\kappa_{1})^{-2}\left(\frac{\partial K^{*}}{\partial z_{1}}\right)^{2} + 
(1-\rho_{*}\kappa_{2})^{-2}\left(\frac{\partial K^{*}}{\partial z_{2}}\right)^{2}
\\
&=& (1+\rho_{*}\kappa^{*}_{1})^{2}\left(\frac{\partial K^{*}}{\partial z_{1}}\right)^{2} + 
(1+\rho_{*}\kappa^{*}_{2})^{2}\left(\frac{\partial K^{*}}{\partial z_{2}}\right)^{2}.
\end{eqnarray*}
Notice now that
$$
\left(\frac{\partial K^{*}}{\partial z_{j}}\right)^{2} = \left(\frac{\partial \kappa_{1}^{*}}{\partial z_{j}}\right)^{2}(\kappa_{2}^{*}-\kappa_{1}^{*})^{2}\ \mbox{ for } j = 1, 2,
$$
since $K^{*}=\kappa^{*}_{1}\kappa^{*}_{2}$ and $\kappa^{*}_{1}+\kappa^{*}_{2} = \frac {1-c}{R\sqrt{c}}$.

Therefore, combining these with \eqref{pre-version}  gives the formula \eqref{formula-gradient_of_Gaussian_curvature}, which completes the proof of 
Proposition \ref{prop:asymptotics-2}.
\end{proof}


\setcounter{equation}{0}
\setcounter{theorem}{0}

\section{Classification of stationary isothermic surfaces in $\mathbb{R}^3$}
\label{section4}

We present two proofs of Theorem \ref{th:uniformly dense}, each one with its own interest.

\vskip.2cm

\begin{proofA}
 Based on Propositions \ref{prop:constant mean curvature}, \ref{prop:asymptotics} and \ref{prop:asymptotics-2}, this proof relies on the theories of properly embedded minimal surfaces and properly embedded constant mean curvature  surfaces in $\mathbb R^3$ and the theory of transnormal functions and transnormal systems.

First of all,  we note that, being parallel to $\partial\Omega$, both $\Gamma$ and $\Gamma_{*}$ are unbounded and connected, which are properties  
they inherit from $\partial\Omega$.

 Since $\Omega$ is uniformly dense in $\Gamma$, by Proposition \ref{prop:asymptotics}, there exists a constant $d$ such that 
\begin{equation}
\label{the second coeff is constant}
h(K) + g=d \ \mbox{ on } \ \partial\Omega.
\end{equation}
Moreover, since 
$$
  H^*=\frac{1-c}{2\, R\sqrt{c}}, \quad \rho_*=\frac R{1+\sqrt{c}},  \ \mbox{ and }  \  K= \frac { K^{*}}{1+2\,\rho_{*}\,H^*+\rho_*^2 K^*},
$$
after a few straightforward computations, Propositions \ref{prop:asymptotics} and \ref{prop:asymptotics-2} give that
 \begin{equation}
 \label{transnormal Gaussian curvature}
 \| \nabla K^{*}\|^{2}= \Psi(K^{*})\ \mbox{ on } \Gamma_{*}\ \mbox{ and } \Psi\left((H^*)^{2}\right) = 0,
 \end{equation}
where $\Psi$ is a  polynomial with coefficients depending only on $c, R, $ and $d$, and the degree of $\Psi$ is at most $4$. 
 
We distinguish two cases: 
$$
\mbox{ (A)\ $K^{*}$ is constant on $\Gamma_{*}$;\quad (B)\ $K^{*}$ is not constant on $\Gamma_{*}$}.
$$

In case (A), since also $H^*$ is constant on $\Gamma_{*}$,  then
$\kappa_1^*$ and $\kappa_2^*$ are both constant on $\Gamma_{*}$ and hence $\Gamma_{*}$ must be either a plane or a circular cylinder, by a classical result. Thus, the conclusion of Theorem \ref{th:uniformly dense} holds true, since  both $\partial\Omega$ and $\Gamma$ are parallel to $\Gamma_{*}$. 

In case (B), the first equation in \eqref{transnormal Gaussian curvature} shows that the Gauss curvature $K^{*}$ is a transnormal function on the connected complete Riemannian manifold $\Gamma_{*}$ and it induces a transnormal system $\mathcal F$ (see Wang \cite{W1987MathAnn}, Miyaoka \cite{Mi2013DiffGeoAppl}, and Bolton \cite{B1973QJMathOxford}). To be precise, in our case, each component of the level sets of $K^{*}$ is called either ``a foil'' or ``a singular foil'' if its dimension is either $1$ or $0$ respectively, and all the components of the level sets of $K^{*}$ generate $\mathcal F$.  All the foils are parallel to each other and any geodesic normal to a foil is orthogonal to every foil. Here, every foil must be a regular curve properly embedded in $\Gamma_{*}$, and every singular foil must be a point in $\Gamma_{*}$ which is a component of the focal varieties (possibly empty) $V_{+}=\{ x \in \Gamma_{*} : K^{*}(x) = \max K^{*} \}$ and $V_{-}= \{ x \in \Gamma_{*} : K^{*}(x) = \min K^{
 *} \}$. 
 
Since a regular curve properly embedded in $\Gamma_{*}$ is either a closed curve or a curve with infinite length and $\Gamma_{*}$ is unbounded, we have that 
\begin{equation}
\label{simple topology}
\Gamma_{*} \mbox{ is homeomorphic to either } \mathbb S^{1}\times \mathbb R\ \mbox{ or } \mathbb R^{2}. 
\end{equation}
This result was proved by Miyaoka in~\cite[Theorem 1.1]{Mi2013DiffGeoAppl}. For instance, if there exists a singular foil, then every foil in a neighborhood of it must be a closed curve and eventually $\Gamma_{*}$ must be homeomorphic to $\mathbb R^{2}$, and if there is no singular foil and one foil is a closed curve, then $\Gamma_{*}$ must be homeomorphic to $\mathbb S^{1}\times \mathbb R$.

Accordingly, it suffices to prove that (B) contradicts the fact that $\Gamma_*$ has constant mean curvature, as guaranteed by Proposition \ref{prop:constant mean curvature}. We arrive at this conclusion by examining two possibilities.


(I) If $\Gamma_*$ has non-zero constant mean curvature (that is when the constant $c$ in \eqref{monge-ampere} is different from $1$), Proposition \ref{prop:constant mean curvature}, together with \eqref{simple topology},  shows that $\Gamma_*$ is properly embedded and of finite topology in $\mathbb R^3, $ it is homeomorphic to either $\mathbb S^2 \setminus \{ \mathcal N, \mathcal S\}$ or $\mathbb S^2 \setminus \{ \mathcal N\}$ (here, $\mathcal N$ and $\mathcal S$ denote the north and south poles of the sphere $\mathbb S^2$),  and each of its ends corresponds to each pole. Then, a theorem due to Meeks \cite[Theorem 1, p.540]{M1988jdg} shows that $\Gamma_*$ is homeomorphic to $\mathbb S^2 \setminus \{ \mathcal N, \mathcal S\}$ and, moreover, a theorem due to Korevaar-Kusner-Solomon \cite[Theorem 2.11, p. 476]{KKS1989jdg} shows that
 $\Gamma_*$ must be either a circular cylinder or an unduloid. See also Kenmotsu \cite[p. 46]{K2003transMathMono221} for an unduloid and \cite{KK1993ProcSymposiaPureMath} for a survey of  properly embedded surfaces in $\mathbb R^3$ with constant mean curvature.  

Since $K^{*}$ is not constant on $\Gamma_{*}$ by assumption (B), \emph{we have that $\Gamma_*$  is an unduloid}, and hence
$\partial\Omega$ is parallel to an unduloid, by Proposition \ref{prop:constant mean curvature}.   Thus, we can choose two points
 $P, Q \in \partial\Omega$ such that 
\begin{equation}
\label{two  points}
 K(P) = K_+>0\mbox{ and }\  K(Q) = K_-<0;
\end{equation}
$P$ and $Q$ lie on $\partial\Omega$ at the maximum (minimum) distance from the common axis of $\partial\Omega$ and $\Gamma_*$.  The symmetry of $\partial\Omega$ ensures that the function $g$ in Proposition \ref{prop:asymptotics} vanishes at $P$ and $Q$, and hence we obtain that
 \begin{equation}
 \label{Gaussian curvatures for unduloid case}
 h(K(P)) + g = h(K_+)\ \mbox{ and }\ h(K(Q)) + g = h(K_-).
 \end{equation}
 
On the other hand, by the intermediate value theorem, there are points $P^{\pm}_* $ in $\partial\Omega$ with $0 < K(P^{+}_*) < K_+$ and $K_{-} < K(P^{-}_{*}) < 0$. Since $h(K)=R^4K^2+2(c-1)R^2K+h(0)$ and $g \le 0$, we obtain that
 $ h(K(P^{+}_*)) + g \le h(K(P^{+}_*)) < h(K_+)$, if $c > 1$, and
 $h(K(P^{-}_*)) + g \le h(K(P^{-}_*)) < h(K_-)$, if $0< c < 1$.
 These contradict \eqref{Gaussian curvatures for unduloid case} because of \eqref{the second coeff is constant}.
 
 
(II) If $\Gamma_*$ has zero mean curvature (that is when $c=1$), again we can claim that $\Gamma_*$  is a properly embedded and of finite topology in $\mathbb R^3$ and  that it is homeomorphic to either $\mathbb S^2 \setminus \{ \mathcal N, \mathcal S\}$ or $\mathbb S^2 \setminus \{ \mathcal N\}$ with each of its ends corresponding to each pole. 

Thus,  if $\Gamma_*$ is homeomorphic to $\mathbb S^2 \setminus \{ \mathcal N, \mathcal S\}$, either by combining results of Schoen \cite{Sc1983jdg} and Collin \cite[Theorem 2, p. 2]{C1997annals} or by combining results of L\'opez and Ros \cite{LR1991jdg} and Collin \cite[Theorem 2, p. 2]{C1997annals} we get that $\Gamma_*$ must be a catenoid.  Instead, if $\Gamma_*$ is homeomorphic to $\mathbb S^2 \setminus \{ \mathcal N\}$, a theorem of Meeks III and Rosenberg \cite[Theorem 0.1, p. 728]{MR2005annals} implies that $\Gamma_*$ must be either a plane or a helicoid.
See also \cite{MP2012univlecAMS} and \cite{CM2011gradst121AMS} for a survey on the minimal surface theory in $\mathbb R^3.$ 
Thus, since $K^{*}$ is not constant, \emph{$\Gamma_*$ must be either a catenoid or a helicoid.}

Now, recall that for $c=1$ we have that
\begin{equation}
\label{curvature_relation_for_minimal_surfaces}
K= \frac {K^*}{1+\rho_*^2K^*}.
\end{equation}

Assume that $\Gamma_*$ is a catenoid; 
then we know that $K^*\le 0$
and $\kappa_j^* \to 0$ as $|x| \to \infty$ for $j = 1,2,$
and hence, by \eqref{curvature_relation_for_minimal_surfaces}, we infer that
$$
K \le 0\ \mbox{ and } \kappa_j \to 0\ \mbox{ as } |x| \to \infty\ (j = 1,2).
$$

Then, with the aid of the interior  estimates  for the minimal surface equation (\cite[Corollary 16.7, p. 407]{GT1983springer}) and Schauder's interior estimates for higher order derivatives (\cite[Problem 6.1. (a), p. 141]{GT1983springer}), by proceeding as in \cite[Proof of Theorem 4.1, pp. 4833-4834]{MPS2006tams}, we see that, for any $k\in\mathbb N$, the $k-$th order derivatives of the function $\varphi$ in Proposition \ref{prop:asymptotics} converge to zero as $|x| \to \infty$; thus, it follows that
 \begin{equation}
 \label{flat}
 h(K) + g \to h(0) \ \mbox{ as } |x| \to \infty.
 \end{equation}
 On the other hand, since $\partial\Omega$ is parallel to the catenoid $\Gamma_*$, we choose a point $P_0 \in \partial\Omega$, which is one of the points nearest to the common axis of $\partial\Omega$ and $\Gamma_*$,  and which  satisfies $\displaystyle K(P_0) =\inf_{P \in \partial\Omega} K(P) < 0$.   Again, the symmetry of $\partial\Omega$ ensures that the function $g$  in Proposition \ref{prop:asymptotics} vanishes at $P_0$ and we conclude that
 $$
 h(K(P_0)) + g = h(K(P_0)) = R^4K(P_0)^2 + h(0) > h(0),
 $$
 that contradicts \eqref{flat} because of \eqref{the second coeff is constant}.
 
 Assume now that $\Gamma_*$ is a helicoid.
 Note that $K^*$ attains its negative minimum on the axis $\ell$ of $\Gamma_*$  and $K^*$ together with the principal curvatures $\kappa_1^*, \kappa_2^*$ tend to zero as the point goes away from $\ell$, as shown in \cite[Example 3.46 (Helicoid), p. 91]{MR2005gradst69AMS}. The same example and  \eqref{curvature_relation_for_minimal_surfaces} imply that $K$ attains its negative minimum on the helix $\tilde \ell$ in $\partial\Omega$ corresponding to $\ell$, and $K$ together with the principal curvatures $\kappa_1, \kappa_2$ tend to zero as the point goes away from $\tilde\ell$. Therefore, 
by the same argument used in the case of the catenoid, as the point on $\partial\Omega$ goes away from $\tilde\ell$,  we obtain that
\begin{equation}
 \label{flat helicoid}
h(K) + g \to h(0) +0 = h(0).
\end{equation}

On the other hand, if we choose a point $\xi_0 \in \tilde\ell$ corresponding to a point $x_0^* \in \ell$, since $K^*$ attains its negative minimum on $\ell$, at $x_0^*$ we have:
$$
\nabla K^*= 0\ \mbox{ and } \ \kappa_1^*\not= \kappa_2^*;
$$
this, together with \eqref{formula-gradient_of_Gaussian_curvature},  yields that $g = 0$ at $\xi_0 \in \tilde\ell$.
Therefore, it follows  that 
 \begin{equation*}
 \label{on the helix}
 h(K(\xi_0))  + g = R^4(\min K)^2 + h(0) +0 > h(0),
 \end{equation*}
that contradicts \eqref{flat helicoid}. The proof is complete. 
\end{proofA}

\vskip.3cm

The proof of Theorem~\ref{th:uniformly dense} presented above is divided in two steps. First, by using the transnormal condition~\eqref{transnormal Gaussian curvature} and the theories of CMC and minimal surfaces, we proved that either $\Gamma_*$ has constant Gauss curvature or it is globally isometric to an unduloid, a catenoid or a helicoid. Second, using the symmetries of the unduloid, the catenoid and the helicoid, and appropriate Schauder estimates (see the proof between Eqs.~\eqref{curvature_relation_for_minimal_surfaces} and~\eqref{flat}), we showed that $\Gamma_*$ cannot be isometric to any of these surfaces. This second step of the proof makes use of general arguments that may be useful in other contexts, but we would like to remark that there is an elementary proof using the explicit expressions of the unduloid, catenoid and helicoid. 

\vskip.2cm

\begin{proofB}
The proof proceeds by inspection, we just check that the transnormal condition~\eqref{transnormal Gaussian curvature}, which is coordinate independent, does not hold on these surfaces. 
\vskip.1cm
\noindent \emph{{\rm(1)} The unduloid.} The family of the unduloids can parametrized using coordinates $(u,v)\in \mathbb R/(2\pi\mathbb Z)\times \mathbb R$, and two real parameters $b>a>0$, see e.g.~\cite{HMO07}. In these coordinates the induced metric reads as
$$
\mathfrak g=\frac12\Big(a^2+b^2+(b^2-a^2)\,\sin\frac{2v}{a+b}\Big)\,du^2+dv^2\,,
$$
and the Gauss curvature is
$$
K=\frac{1}{(a+b)^2}-\frac{4a^2b^2}{(a+b)^2}\Big(a^2+b^2+(b^2-a^2)\,\sin\frac{2v}{a+b}\Big)^{-2}\,.
$$    
A straightforward computation using the metric $\mathfrak g$ and the curvature $K$ yields that
$$
\|\nabla K\|^2=(K_v)^2=\Big[1-(a+b)^2K\Big]^2\Big\{A_1+A_2\,[1-(a+b)^2K]^{1/2}+A_3\,[1-(a+b)^2K]\Big\}\,,
$$
where $A_1,A_2,A_3$ are real constants that can be written explicitly in terms of $a$ and $b$, but whose expressions are not relevant for our purposes. It can be checked that $A_2\neq 0$ for any values of $a$ and $b$, thus implying that $\|\nabla K\|^2$ is not a polynomial of $K$, and hence the surface $\Gamma_*$ cannot be an unduloid on account of Eq.~\eqref{transnormal Gaussian curvature}. 
\vskip.1cm
\noindent \emph{{\rm(2)} The catenoid.} The family of the catenoids can be parametrized using coordinates $(u,v)\in \mathbb R/(2\pi\mathbb Z)\times \mathbb R$ and a real constant $a>0$, cf.~\cite{MP2012univlecAMS} and~\cite{CM2011gradst121AMS}. In this coordinate system the induced metric and Gauss curvature are
$$
\mathfrak g=a^2\cosh^2\Big(\frac{v}{a}\Big)\,du^2+\cosh^2\Big(\frac{v}{a}\Big)\,dv^2\,, \quad
K=\frac{-1}{a^2\cosh^4(v/a)}\,.
$$
As before, after some computations the quantity $\|\nabla K\|^2$ can be written in terms of $K$ as
$$
\|\nabla K\|^2=\frac{(K_v)^2}{\cosh^2(v/a)}=\frac{16}{a}(-K)^{5/2}+16K^3\,,
$$
which is not a polynomial of $K$. Therefore, $\Gamma_*$ cannot be a catenoid.
\vskip.1cm
\noindent \emph{{\rm(3)} The helicoid.} The family of the helicoids can be parametrized with coordinates $(u,v)\in\mathbb R^2$ and a real constant $a>0$, see~\cite{MR2005gradst69AMS}. The induced metric and Gauss curvature read in these coordinates as
$$
\mathfrak g=(a^2+v^2)\,du^2+dv^2\,, \quad K=\frac{-a^2}{(a^2+v^2)^2}\,.
$$
Algebraic calculations again yield that $\|\nabla K\|^2$ is not a polynomial of $K$,  in fact,
$$
\|\nabla K\|^2=(K_v)^2=\frac{16}{a}(-K)^{5/2}+16K^3\,.
$$
Therefore, $\Gamma_*$ cannot be a helicoid either. 
\end{proofB}

\begin{remark} {\rm We can also show that, in the proof of Theorem \ref{th:uniformly dense},  if $K^{*}$ is not constant, then it is an isoparametric function, namely it satisfies the system of equations
$$
\Vert\nabla K^*\Vert^2=\Psi(K^*) \ \mbox{ and } \ \Delta_{\Gamma_{*}} K^{*} = \Phi(K^{*})\ \mbox{ on } \Gamma_{*},
$$
for some continuous function $\Phi$; here, $\Delta_{\Gamma_{*}}$ is the Laplace-Beltrami operator on $\Gamma_{*}$. In our case, $\Phi$ and $\Psi$ are polynomials. 
\par
In fact, the umbilical points of the surface $\Gamma_{*}$ of constant mean curvature are isolated (see \cite[Proposition 1.4 and (1.40), p. 21]{K2003transMathMono221} ), and by \cite[(1.41), p. 22]{K2003transMathMono221} 
$$
\Delta_{\Gamma_{*}} \log\sqrt{(H^*)^{2}-K^{*}} -2 K^{*} = 0\ \mbox{ on } \Gamma_{*} \setminus \{ \mbox{ umbilical points }\}.
$$
Therefore, it follows from the first equation of \eqref{transnormal Gaussian curvature} that
$$
\Delta_{\Gamma_{*}} K^{*} = -4K^{*}\,[(H^*)^{2}-K^{*}] - \frac 1{(H^*)^{2}-K^{*}}\,\Psi(K^{*}).
$$
Thus, the second equality of \eqref{transnormal Gaussian curvature} guarantees that the right-hand side of this equation is written as $\Phi(K^{*})$ for some polynomial  $\Phi = \Phi(t)$ in $t \in \mathbb R$.
}
\end{remark}

\setcounter{equation}{0}
\setcounter{theorem}{0}

\section{Uniformly dense domains in $\mathbb R^3$: the case $\Gamma=\partial\Omega$}
\label{section5}

With the aid of Nitsche's result \cite{N1995analysis}, the theory of embedded minimal surfaces of finite topology in $\mathbb R^3$ (\cite{BB2011cmhelv}, \cite{MR2005annals}, \cite{C1997annals}) gives the following generalization of \cite[Theorem 1.4, p. 4824]{MPS2006tams}:
\begin{theorem}
\label{th:finite topology uniformly dense minimal surface} 
Let $S$ be a complete embedded minimal surface of finite topology in $\mathbb R^3$, and let $\Omega$ be one connected component of $\mathbb R^3\setminus S$.

 If $\Omega$ is uniformly dense in $S\ (=\partial\Omega)$,
 then $S$ must be either a plane or a helicoid.
\end{theorem}

 \noindent
\begin{proof} First of all, we note that $S$ must be properly embedded in $\mathbb R^3$ by Colding and Minicozzi II \cite[Corollary 0.13, p. 214]{CM2008annals}, and hence $S$ separates $\mathbb R^3$ into two connected  components.

We shall use an argument similar to those used in  \cite[Proof of Theorem 1.4, p. 4833--4834]{MPS2006tams}. Since $S$ is of finite topology, 
 there exist a compact Riemann surface $M$ without boundary in $\mathbb R^{3}$ and a finite number of points $p_1, \dots, p_m \in M$ such that $S$ is homeomorphic to  $M \setminus \{ p_1, \dots, p_m\}$ and each end corresponds to each $p_{j}$. Then the structure theorem of 
Bernstein-Breiner \cite{BB2011cmhelv} (see also Meeks III-Rosenberg \cite{MR2005annals} and Collin \cite{C1997annals}) shows the following:
\begin{itemize}
\item[(i)]\ If $m \ge 2$, then $S$ has finite total curvature and each end of $S$ is asymptotic to either a plane or a half catenoid.  See \cite[Corollary 1.4, p. 357]{BB2011cmhelv} and \cite[Theorem 2, p. 2]{C1997annals};
\item[(ii)]\ If $m=1$, then either $S$ is a plane or it has infinite total curvature and its end is asymptotic to a helicoid. See \cite[Corollary 1.4, p. 357]{BB2011cmhelv} and \cite[Theorems 0.1 and 0.2, p. 728]{MR2005annals};
\item[(iii)]\ The Gauss curvature of $S$ is bounded, and hence the principal curvatures of $S$ are also bounded. See \cite[Proposition 1, p. 801]{Sc1983jdg} and \cite[Theorem 1, p. 1336]{HPR2001tams} together with \cite{BB2011cmhelv} and \cite{MR2005annals}.
\end{itemize}
See also \cite{MP2012univlecAMS} and \cite{CM2011gradst121AMS} for the minimal surface theory in $\mathbb R^3.$

Now, item (iii) above guarantees that there exists $\delta > 0$ such that, 
for every $x \in S$, the connected component of $B_{\delta}(x) \cap S$ containing $x$ is written as a graph of a function over the tangent plane to $S$ at $x$ (see \cite[Lemma 2.4, p. 74]{CM2011gradst121AMS} for a proof). Hence combining the above (iii) with the interior  estimates  for the minimal surface equation (see \cite[Corollary 16.7, p. 407]{GT1983springer}) yields that the convergence in (i) and (ii) is in the $C^k$ local topology for any $k \in \mathbb N$. 

Therefore, in view of the geometry of a hyperplane, a half catenoid, and a helicoid, each of (i) and (ii) gives a sequence of points $\{ P_j\}$ in $S$ such that the principal curvatures of the connected component of $B_{\delta}(P_j) \cap S$ containing $P_j$ tend to zero uniformly as $j \to \infty$.  Thus we can apply \cite[Theorem 4.1, p. 4833]{MPS2006tams}, which uses Nitsche's result \cite{N1995analysis},  to complete the proof of Theorem \ref{th:finite topology uniformly dense minimal surface}. \end{proof}

In terms of stationary isothermic surfaces, and using~\cite[Theorems 1.1 and 1.3]{MPS2006tams}, Theorem~\ref{th:finite topology uniformly dense minimal surface} implies the following corollary:

\begin{corollary}
Let $\Omega$ be a domain in $\mathbb R^3$ whose boundary $\partial\Omega$ is an unbounded complete embedded surface. Assume that $\partial\Omega$ has finite topology and is a stationary isothermic surface of the solution $u$ of the Cauchy problem~\eqref{cauchy}. Then $\partial\Omega$ must be a plane, a circular cylinder or a helicoid. 
\end{corollary}


\setcounter{equation}{0}
\setcounter{theorem}{0}

\def\theequation{A.\arabic{equation}}
\def\thetheorem{A.\arabic{theorem}}

\appendix
\section*{Appendix}
\label{appendix}
The following list of definite integrals will be used in the calculations of Lemmas~\ref{le: the first four terms in 3rd}--\ref{le: the sum of the last two terms in 3rd in terms of c and K}.
They easily follow by means of successive differentiations and algebraic manipulations of the formula:
\begin{equation}
\label{starting point}
\frac1{2\pi}\,\int_0^{2\pi} \frac{d\theta} {\sigma_1\cos^2\theta + \sigma_2\sin^2\theta}= \sigma_1^{-\frac12}\sigma_2^{-\frac12};
\end{equation}
here, $\sigma_1$ and $\sigma_2$ are two positive parameters. For $0\le j\le m$ and $m=0,1,\dots$ we have:
\begin{equation}
\label{mj}
\frac1{2\pi}\,\int_0^{2\pi} \frac{(\cos\theta)^{2j} (\sin\theta)^{2m-2j}}{(\sigma_1\cos^2\theta + \sigma_2\sin^2\theta)^{m+1}} d\theta= \frac1{2^{2m}}\frac{(2j)! (2m-2j)!}{m! j! (m-j)!}\, \sigma_1^{-\frac12-j}\sigma_2^{-\frac12-(m-j)};
\end{equation}
for $0\le j\le m$ and $m=0,1,\dots$;
\begin{equation}
\label{m+1}
\frac1{2\pi}\,\int_0^{2\pi} \frac{d\theta}{(\sigma_1\cos^2\theta + \sigma_2\sin^2\theta)^{m+1}} = \frac1{2^{2m}}\,\sum_{j=0}^m {2j\choose j} {2(m-j)\choose m-j}\,\sigma_1^{-\frac12-j}\sigma_2^{-\frac12-(m-j)};
\end{equation}
\begin{multline}
\label{m+1 - bis1}
\frac1{2\pi}\,\int_0^{2\pi} \frac{\cos^2\theta\,d\theta}{(\sigma_1\cos^2\theta + \sigma_2\sin^2\theta)^{m+2}} = \\
\frac1{2^{2m+1}}\,\sum_{j=0}^m \frac{2j+1}{m+1}{2j\choose j} {2(m-j)\choose m-j}\,\sigma_1^{-\frac12-j-1}\sigma_2^{-\frac12-(m-j)}
\end{multline}
and
\begin{multline}
\label{m+1 - bis2}
\frac1{2\pi}\,\int_0^{2\pi} \frac{\sin^2\theta\,d\theta}{(\sigma_1\cos^2\theta + \sigma_2\sin^2\theta)^{m+2}} = \\
\frac1{2^{2m+1}}\,\sum_{j=0}^m \frac{2m-2j+1}{m+1}{2j\choose j} {2(m-j)\choose m-j}\,\sigma_1^{-\frac12-j}\sigma_2^{-\frac12-(m-j+1)}.
\end{multline}

In this paper we  use \eqref{mj} for $1\le m\le 3$, \eqref{m+1} for $1 \le m \le 2$, \eqref{m+1 - bis1} and \eqref{m+1 - bis2} for $m=1$, respectively.
In the sequel, we  set $\kappa_j = \kappa_j(\xi)$, for $j = 1, 2$, and abbreviate the  
partial derivatives of $\varphi$ with respect to $z_1$ and $z_2$ by subscripts; whenever it is needed, we shall specify their arguments: the varying point $z=(z_1, z_2)$ or the origin $(0,0)$.

The following two lemmas are preparatory for Lemmas~\ref{le: the first four terms in 3rd in terms of c and K} and~\ref{le: the sum of the last two terms in 3rd in terms of c and K} below.


\begin{lemma} The following formulas hold:
\label{le: the first four terms in 3rd}

$$
-\frac{2\,c^\frac32}{\pi}\int_0^{2\pi}\frac {P_2(v)\, d\theta}{(\sigma_1\cos^2\theta + \sigma_2\sin^2\theta)^2}=\kappa_1\,\sigma_2+ \kappa_2\,\sigma_1;
$$

$$
-\frac{8\,c^\frac52}{\pi}\,\int_0^{2\pi}\frac {P_2(v)\, d\theta}{(\sigma_1\cos^2\theta + \sigma_2\sin^2\theta)^3}= \kappa_1\, \bigl(3\sigma_2^2+ \sigma_1\sigma_2\bigr)
+\kappa_2\, \bigl(\sigma_1\sigma_2+ 3\sigma_1^2\bigr);
$$
\begin{multline*}
 \frac{2^5 3^2 c^\frac72}{\pi}\int_0^{2\pi}\frac {[P_3(v)]^2\,d\theta}{(\sigma_1\cos^2\theta + \sigma_2\sin^2\theta)^4} = \\
5\,(\varphi_{111})^2 \sigma_2^{3} + 9\, (\varphi_{112})^2 \sigma_1\sigma_2^{2}
+ 
9\, (\varphi_{122})^2 \sigma_1^{2}\sigma_2
+ 5\, (\varphi_{222})^2 \sigma_1^{3}+ \\
6\,(\varphi_{111})(\varphi_{122})\, \sigma_1\sigma_2^2+ 
6\,(\varphi_{112})(\varphi_{222})\, \sigma_1^{2}\sigma_2;
\end{multline*}
$$
\frac{2^5\,c^{\frac52}}{\pi}\,\int_0^{2\pi}\frac {P_4(v)\,d\theta }{(\sigma_1\cos^2\theta + \sigma_2\sin^2\theta)^3} =(\varphi_{1111})\,\sigma_2^{2}+ 2\,(\varphi_{1122})\,\sigma_1\sigma_2
+ (\varphi_{2222})\,\sigma_1^{2}.
$$
Here, we mean that the derivatives of $\varphi$ are evaluated at $(0,0)$.
\end{lemma}

\begin{proof} Since $-2\,P_2(v) =\kappa_1\,\cos^2\theta + \kappa_2\,\sin^2\theta$,  with the aid of \eqref{monge_ampere constant c},
the first formula  follows from \eqref{mj} for $m=1$, and the second one follows from \eqref{m+1 - bis1} and \eqref{m+1 - bis2} for $m=1$.


Observe that 
\begin{multline*}
36\,[P_3(v)]^{2} = \left\{\left(\cos\theta\,\partial_1+ \sin\theta\,\partial_2\right)^{3}\varphi\right\}^{2}
\\
= (\varphi_{111})^{2}\cos^6\theta + 9\,(\varphi_{112})^2 \cos^4\theta\sin^{2}\theta
+ 9\,(\varphi_{122})^2\cos^2\theta\sin^{4}\theta
+ (\varphi_{222})^{2}\sin^6\theta
\\
+ 6\,(\varphi_{111}) (\varphi_{122}) \cos^{4}\theta\sin^{2}\theta
+ 6\,(\varphi_{112})(\varphi_{222}) \cos^{2}\theta\sin^{4}\theta
\\
+ \left[\mbox{ the sum of odd functions of either $\cos\theta$ or $\sin\theta$ } \right],
\end{multline*}
and
\begin{multline*}
24\,P_4(v) =  \left(\cos\theta\,\partial_1 + \sin\theta\,\partial_2\right)^{4}\!\varphi
\\
=(\varphi_{1111}) \cos^4\theta  +
6\,(\varphi_{1122}) \cos^{2}\theta\sin^{2}\theta 
+(\varphi_{2222}) \sin^4\theta 
\\
+ \left[\mbox{ the sum of odd functions of either $\cos\theta$ or $\sin\theta$ } \right].
\end{multline*}
Then, with the aid of \eqref{monge_ampere constant c},  the third and fourth formulas follow  from \eqref{mj} with $m=3$ and $m=2$ respectively.
\end{proof}


\begin{lemma} 
\label{le:phiderivatives}
Let $\varphi$ be the function representing $\partial\Omega$ locally as in {\rm Proposition \ref{prop:asymptotics}}. If $\Omega$ is uniformly dense in $\Gamma$, then
\begin{equation}
\label{3rd original derivatives at 0-1}
\sigma_2\varphi_{111} + \sigma_1 \varphi_{122}= 0, \quad \sigma_1\varphi_{222} + \sigma_2 \varphi_{112} = 0,
\end{equation}
\begin{multline}
\sigma_1^{-1}\varphi_{1111}  + \sigma_2^{-1}\varphi_{1122}  = -\frac{2R}{c}\left\{ \varphi_{111}\varphi_{122}-(\varphi_{112})^2 \right\}
\\
+ \frac 1{Rc}\left[ 4(c-1)\kappa_1^2 + R\kappa_1^2(\kappa_1+3\kappa_2)\right],
\label{4th derivatives at 0-1}
\end{multline}
and
\begin{multline}
\sigma_2^{-1}\varphi_{2222}+ \sigma_1^{-1}\varphi_{1122}=-\frac{2R}{c}\left\{ \varphi_{222}\varphi_{112}-(\varphi_{122})^2 \right\}
\\
+ \frac 1{Rc}\left[ 4(c-1)\kappa_2^2 + R\kappa_2^2(\kappa_2+3\kappa_1)\right].
\label{4th derivatives at 0-2}
\end{multline}
Here, $\sigma_1$ and $\sigma_2$ are given by \eqref{sigma_j} and the derivatives of $\varphi$
are evaluated at $(0,0)$. 
\end{lemma}

\begin{proof}
Since \eqref{monge-ampere} gives
\begin{equation}
\label{Gaussian and Mean Curvatures}
 - R(\kappa_1 + \kappa_2) + R^2\kappa_1\kappa_2 = c -1,
\end{equation}
the function $\varphi(z_1,z_2)$ satisfies the partial differential equation:
\begin{multline}
R\sqrt{1+\varphi_1^2+\varphi_2^2}\left\{ (1+\varphi_2^2)\varphi_{11} - 2\varphi_1\varphi_2\varphi_{12} + (1 + \varphi_1^2)\varphi_{22}\right\} + \\
R^2 \left\{ \varphi_{11}\varphi_{22} - (\varphi_{12})^2\right\} 
= (c-1)( 1+\varphi_1^2 + \varphi_2^2 )^2,
\label{monge-ampere graph eq}
\end{multline}
for $z$ in a neighborhood of $(0,0)$.


Recall that at $(0,0)$
\begin{equation}
\label{values of varphi at zero}
\varphi_1= \varphi_2= \varphi_{12}= 0\ \mbox{ and } \varphi_{jj}= -\kappa_j(\xi)\mbox{ for } j = 1, 2.
\end{equation}
By differentiating \eqref{monge-ampere graph eq} with respect to $z_1$, we obtain
\begin{multline}
R\frac{\varphi_1\varphi_{11} + \varphi_2\varphi_{12}}{\sqrt{1+\varphi_1^2+\varphi_2^2}}\left\{ (1+\varphi_2^2)\varphi_{11} - 2\varphi_1\varphi_2\varphi_{12}  + (1 + \varphi_1^2)\varphi_{22}\right\}
\\
+R\sqrt{1+\varphi_1^2+\varphi_2^2}\left\{2\varphi_2\varphi_{12}\varphi_{11} +  (1+\varphi_2^2)\varphi_{111} - 2\varphi_{11}\varphi_2\varphi_{12} - 2\varphi_1 (\varphi_{12})^2 - 2\varphi_1\varphi_2\varphi_{112} \right.
\\
\left.+2\varphi_1\varphi_{11}\varphi_{22} +  (1 + \varphi_1^2)\varphi_{122}\right\}
+R^2(\varphi_{111}\varphi_{22} + \varphi_{11}\varphi_{122} -2\varphi_{12}\varphi_{112} )
\\
= 4(c-1) (1+\varphi_1^2+\varphi_2^2)(\varphi_1\varphi_{11} +  \varphi_2\varphi_{12})\label{first derivative in z1}
\end{multline}
Letting $z = (0,0)$ 
in \eqref{first derivative in z1} 
yields, in view of \eqref{values of varphi at zero}, that
$$
R\,(\varphi_{111}+ \varphi_{122}) -R^2\,(\kappa_2\,\varphi_{111} + \kappa_1\,\varphi_{122})= 0,
$$
and hence the first formula in \eqref{3rd original derivatives at 0-1}.
By differentiating \eqref{monge-ampere graph eq} with respect to $z_2$,
a similar calculation gives the second formula in \eqref{3rd original derivatives at 0-1}.

Again, by differentiating \eqref{first derivative in z1} with respect to $z_1$, then letting $z = (0,0)$ in the resulting equation, and using  \eqref{values of varphi at zero}, we get
\begin{multline*}
-R\,\kappa_1^2(\kappa_1+\kappa_2) + R\left\{\varphi_{1111} -2\kappa_1^2 \kappa_2+ \varphi_{1122}\right\} \\
+R^2\left\{ -\kappa_2\,\varphi_{1111}+ 2\varphi_{111} \varphi_{122} -\kappa_1\,\varphi_{1122}  -2\left(\varphi_{112}\right)^2\right\}
= 4(c-1)\kappa_1^2.
\end{multline*}
Hence,  with the aid of \eqref{sigma_j} and \eqref{monge_ampere constant c},
we obtain \eqref{4th derivatives at 0-1}.
By differentiating \eqref{monge-ampere graph eq} twice with respect to $z_2$ and then letting $z = (0,0)$ in the resulting equation, similar calculations yield \eqref{4th derivatives at 0-2}.
\end{proof}

\vskip.3cm

We now complete the computation of the coefficient of $s^2$ in \eqref{formula-asymptotics};
we must integrate over $[0,2\pi]$ the function in \eqref{3rd coefficient}. 

In view of \eqref{sigma_j}, \eqref{monge_ampere constant c}, \eqref{Gaussian and Mean Curvatures}, we preliminarily note that
\begin{equation}
\label{use c and K}
\sigma_1+\sigma_2= 1+c-R^2K \ \mbox{ and } \kappa_1+\kappa_2= \frac {1-c}R + RK.
\end{equation}

In the following lemma, we use \eqref{m+1} for $m=1, 2$ and the first two formulas in Lemma \ref{le: the first four terms in 3rd};  by using also  \eqref{use c and K},  \eqref{monge_ampere constant c} and some algebraic manipulations, we obtain the integrals of the first four terms in \eqref{3rd coefficient}. 


\begin{lemma} The following formulas hold:
\label{le: the first four terms in 3rd in terms of c and K}
$$
- \int_0^{2\pi} \frac {7\, d\theta}{6R^2(\sigma_1\cos^2\theta + \sigma_2\sin^2\theta)^2} =   \frac{7\pi}{6 R^2 c^{3/2}}\,(R^2K-1-c) ,
$$
$$
\int_0^{2\pi}\frac{d\theta}{6R^2(\sigma_1\cos^2\theta + \sigma_2\sin^2\theta)^3} = \frac{\pi}{24 R^2 c^{5/2}}\,\Bigl[3 R^4K^2- 6(1+c)R^2 K+3c^2+2c+3\Bigr],
$$
$$
 - \int_0^{2\pi}\frac {4P_2(v)\,d\theta}{R(\sigma_1\cos^2\theta + \sigma_2\sin^2\theta)^2} =  -\frac{2\pi}{R^2c^{3/2}}\, (R^2 K + c-1),
$$
$$
\int_0^{2\pi}\frac {4P_2(v)\,d\theta}{3R(\sigma_1\cos^2\theta + \sigma_2\sin^2\theta)^3} = - \frac{\pi}{6 R^2 c^{5/2}}\,\Bigl[3R^4K^2-2(3+c) R^2 K-(c^2+2c-3)\Bigr].
$$
\end{lemma}

\vskip.3cm

We finally obtain the integrals of the last two terms in \eqref{3rd coefficient} by the last two formulas in Lemma \ref{le: the first four terms in 3rd}, Lemma \ref{le:phiderivatives}  and similar algebraic manipulations.

\begin{lemma}The following formula holds:
\label{le: the sum of the last two terms in 3rd in terms of c and K}
\begin{multline*}
\int_0^{2\pi}\frac {12R^2[P_3(v)]^2\,d\theta }{(\sigma_1\cos^2\theta + \sigma_2\sin^2\theta)^4} - \int_0^{2\pi}\frac {4RP_4(v)\, d\theta}{(\sigma_1\cos^2\theta + \sigma_2\sin^2\theta)^3}
 \\
= -\frac {\pi R^2}{6 \sqrt{c}} \Bigl[ (1-R \kappa_1)^{-3}(\varphi_{111})^2 + (1-R \kappa_2)^{-3}(\varphi_{222})^2 \Bigr]
 \\
+ \frac{\pi}{8 R^2 c^{5/2}}\Bigl\{(c+3)[R^4K^2 + 2(c-1) R^2 K]-3(c-1)^3\Bigr\}. 
\end{multline*}
\end{lemma}

\vskip.4cm

\noindent{\large\bf Acknowledgement.}
\smallskip
The authors would like to thank Professor Reiko Miyaoka for her interest in their work and some useful discussions.
Also, the authors are grateful to the anonymous referees for giving invaluable suggestions to improve the paper.

\end{document}